\documentclass[12pt]{amsart} 

\usepackage{amsmath,amsfonts,amssymb}

\usepackage{graphicx} 

\usepackage{epstopdf} 

\usepackage[all,cmtip]{xy} 

\newcommand{\iinfty}{{\mathchoice
{\begin{minipage}{.15in}\includegraphics[width=.15in]{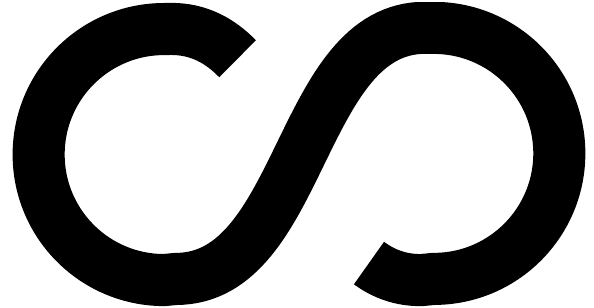}\end{minipage}}
{\begin{minipage}{.10in}\includegraphics[width=.10in]{infty2.pdf}\end{minipage}}
{\begin{minipage}{.08in}\includegraphics[width=.08in]{infty2.pdf}\end{minipage}}
{\begin{minipage}{.08in}\includegraphics[width=.08in]{infty2.pdf}\end{minipage}}
}}

\theoremstyle{plain} 
\newtheorem{thm}{Theorem} 
\newtheorem{cor}[thm]{Corollary} 
\newtheorem{lem}[thm]{Lemma} 
\newtheorem{conj}[thm]{Conjecture} 
\newtheorem{prop}[thm]{Proposition} 
\theoremstyle{definition} 
\newtheorem{defn}[thm]{Definition} 
\newtheorem{rem}[thm]{Remark}

\renewcommand{\int}{\operatorname{int}} 
 
\newcommand{\Ker}{\operatorname{Ker}} 
\newcommand{\Cok}{\operatorname{Cok}} 
\newcommand{\id}{\operatorname{id}}

\newcommand{\im}{\operatorname{Im}}

\newcommand\W{\text{\sf W}} 
\newcommand\G{\text{\sf G}} 
\newcommand\sL{\text{\sf L}}
\newcommand\sD{\text{\sf D}}

\newcommand{\Z}{\mathbb{Z}} 
\newcommand{\N}{\mathbb{N}}

\newcommand{\bG}{\mathbb{G}} 
\newcommand{\bW}{\mathbb{W}} 
\newcommand{\bL}{\mathbb{L}}

\newcommand{\cT}{\mathcal{T}} 
\newcommand{\cV}{\mathcal{V}} 
\newcommand{\cW}{\mathcal{W}}

\newcommand{\sra}{\twoheadrightarrow}

\begin{document} 

\title{Geometric filtrations of \\classical link concordance} 

\begin{abstract} 
This paper studies \emph{grope} and \emph{Whitney tower} filtrations on the set of concordance classes of classical links in terms of \emph{class} and \emph{order} respectively. Using the tree-valued intersection theory of Whitney towers, the associated graded quotients are shown to be finitely generated abelian groups under a well-defined connected sum operation. 

\emph{Twisted Whitney towers} are also studied, along with a corresponding quadratic enhancement of the intersection theory for framed Whitney towers that measures Whitney-disk framing obstructions. The obstruction theory in the framed setting is strengthened, and 
the relationships between the twisted and framed filtrations
are described in terms of exact sequences which show how higher-order Sato-Levine and higher-order Arf invariants are obstructions to framing a twisted Whitney tower. 

The results from this paper combine with those in \cite{CST2,CST3,CST4} to give a classifications of the filtrations; see our survey \cite{CST0} as well as the end of the introduction below. UPDATE: The results of this paper have been completely subsumed into the paper \emph{Whitney tower concordance of classical links} \cite{WTCCL}. 
\end{abstract}

\author[J. Conant]{James Conant} 
\email{jconant@math.utk.edu} 
\address{Dept. of Mathematics, University of Tennessee, Knoxville, TN} 

\author[R. Schneiderman]{Rob Schneiderman} 
\email{robert.schneiderman@lehman.cuny.edu} 
\address{Dept. of Mathematics and Computer Science, Lehman College, City University of New York, Bronx, NY} 

\author[P. Teichner]{Peter Teichner} 
\email{teichner@mac.com} 
\address{Dept. of Mathematics, University of California, Berkeley, CA and} 
\address{Max-Planck Institut f\"ur Mathematik, Bonn, Germany}

\keywords{Whitney towers, gropes, link 
concordance, trees, higher-order Arf invariants, higher-order Sato-Levine invariants, twisted Whitney disks} 

\maketitle

\section{Introduction}\label{sec:intro} 
Several key theorems and conjectures in low-dimensional topology can 
be stated in terms of certain 2-complexes known as \emph{gropes}. 
These are geometric embodiments of commutators of group 
elements, see e.g.~\cite{Can, COT, CT1,CT2,FQ,FT2,T1, T2}. 
Gropes are built by gluing together surfaces along collections of embedded essential circles, and come in 
several different types with varying measures of complexity. 

The {\em height} of a grope corresponds to the derived series and was used (in the presence of caps) in \cite{FQ,FT2} to formulate the main open problem for topological $4$--manifolds: the 4-dimensional surgery and s-cobordism theorems for arbitrary fundamental groups are equivalent to a certain statement about capped gropes of height $\geq 2$. In the uncapped setting, gropes of increasing height were used in \cite{COT} to define a new filtration of the knot concordance group. 

In this paper, we shall study a more basic measure of the complexity, namely the 
\emph{class} of a grope, which corresponds to commutator length and the lower central series (Figure~\ref{fig:class4gropeA}). In the 3-dimensional setting it was used in \cite{CT1,CT2} to give a geometric interpretation of Vassiliev's finite type invariants of knots and (the first nonvanishing term of) the Kontsevich integral. 


\begin{figure}[h]
\centerline{\includegraphics[scale=.45]{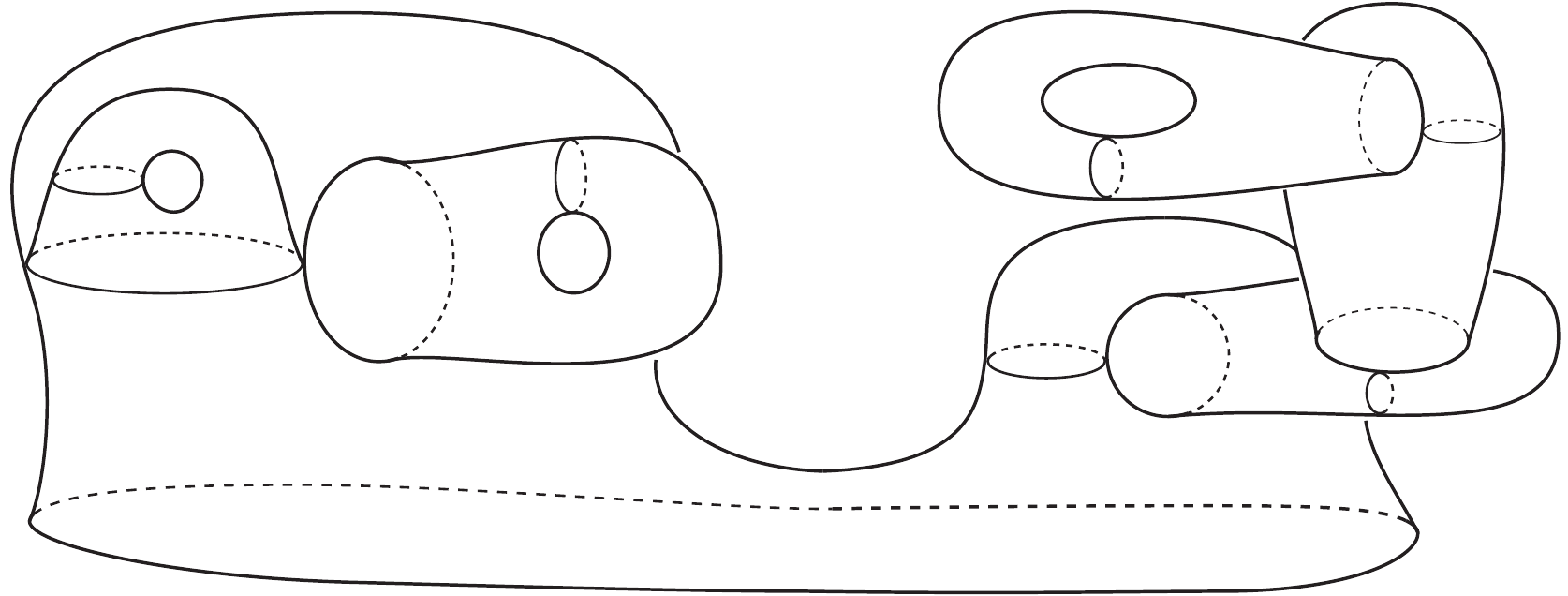}}
         \caption{A class $4$ grope.}
         \label{fig:class4gropeA}
\end{figure}

In the 4-dimensional setting it is a fundamental open problem to determine under what conditions the components of a link in the $3$--sphere bound class $n$ gropes disjointly embedded in the $4$--ball, and this paper will describe a 
program for computing the \\
 \emph{Grope concordance filtration} (by class) 
\begin{equation} \tag{$\bG$} 
\dots \subseteq \bG_{3} \subseteq \bG_{2} \subseteq \bG_{1} \subseteq  \bG_{0} \subseteq  \bL 
\end{equation}
on the set $\bL=\bL(m)$ of framed links in $S^3$ with $m$ components. Here $\bG_n=\bG_n(m)$ is defined to be the set of framed links that bound class~$(n+1)$ framed gropes disjointly embedded in $B^4$. The index shift provides compatibility with other filtrations introduced below and for the same reason, we {\em define} $\bG_0$ to be the set of evenly framed links. 

The intersection of all $\bG_n$ contains all slice links because a 2-disk is a grope of arbitrary large class. In fact, this filtration factors through link concordance; and we shall use this fact implicitly at various places. 

Recall that the connected sum of links is not a well-defined operation because of the choices of connecting bands that are involved. As a consequence, there is no direct definition of the associated graded quotients of the filtration $\bG_n$. However, one can define an equivalence relation on links by using the notion of {\em class~$n$ grope concordance} between two links. This is obtained by using framed gropes built on annuli connecting link components in $S^3\times I$. We then define the {\em associated graded} $\G_n=\G_n(m)$ as the quotient of 
$\bG_n$ modulo grope concordance of class $n+2$. Our first result is the following corollary of Theorem~\ref{thm:R-onto-G} below. 

\begin{cor} \label{cor:gropes} For all $m,n\in\N$, the sets $\G_n(m)$ are finitely generated abelian groups, under a well-defined connected sum $\#$. Moreover, $\G_n$ is the set of framed links $L\in\bG_n$ modulo the relation that $[L_1]=[L_2]\in\G_n$ if and only if $L_1 \# -L_2$ lies in $\bG_{n+1}$, for some choice of connected sum $\#$. Here $-L$ is the mirror image of $L$ with reversed framing. 
\end{cor} 

For example, $\G_0(m) \cong\Z^k$ where $k=m(m+1)/2$ is the number of possible linking numbers and framings (on the diagonal) of a link with $m$ components. This follows from the fact that disjointly embedded class 2 gropes in $B^4$ are framed surfaces which induce the zero framings on their boundary and show the vanishing of all linking numbers.  We also note that $\G_1(1) \cong \Z_2$ is given by the Arf invariant, and that $\bG_2(1) =\bG_n(1)$ for all $n\geq 2$, by \cite{S2}. 

The proof of Corollary~\ref{cor:gropes} can be most succinctly formulated by defining a sequence of finitely generated abelian groups $\cT_n=\cT_n(m)$ in terms of certain trees. These groups have previously appeared in the study of graph cohomology, Feynman diagrams and finite type invariants of links and $3$--manifolds. 

\begin{defn}\label{def:Tau}\cite{ST2}
In this paper, a {\em tree} will always refer to an oriented unitrivalent tree, where the {\em orientation} of a tree is given by cyclic orientations at all trivalent vertices. The \emph{order} of a tree is the number of trivalent vertices.
Univalent vertices will usually be labeled from the set $\{1,2,3,\ldots,m\}$ corresponding to the link components, and we consider trees up to isomorphisms preserving these labelings.
We define $\cT=\cT(m)$ to be the free abelian group on such trees, modulo the antisymmetry (AS) and Jacobi (IHX) relations shown in Figure~\ref{fig:ASandIHXtree-relations}. 
\begin{figure}[h]
\centerline{\includegraphics[scale=.65]{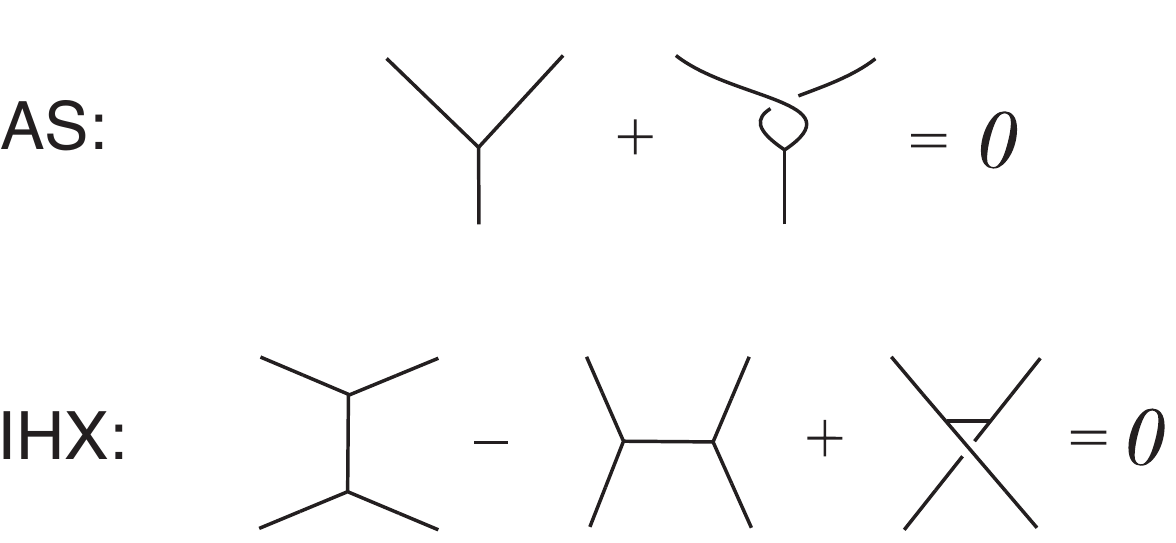}}
         \caption{Local pictures of the \emph{antisymmetry} (AS) and \emph{Jacobi} (IHX) relations
         in $\cT$. Here all trivalent orientations are induced from a fixed orientation of the plane, and univalent vertices possibly extend to subtrees which are fixed in each equation.}
         \label{fig:ASandIHXtree-relations}
\end{figure}
Since the AS and IHX relations are homogeneous with respect to order, $\cT$ inherits a grading $\cT=\oplus_n\cT_n$, where $\cT_n=\cT_n(m)$ is the free abelian group on order $n$ trees, modulo AS and IHX relations.
\end{defn}

Then a construction similarl to Cochran's iterated Bing-doubling construction for realizing Milnor invariants \cite{C} (and equivalent to `simple clasper surgery along trees' in the sense of Goussarov and Habiro)
leads to our {\em realization map} 
\[ 
R_n : \cT_n \to \G_n 
\] 
and Theorem~\ref{thm:R-onto-G} says that this map is well-defined and onto. For example, if $ i\neq j$ then $R_0$ sends the 
order zero tree $i -\!\!\!-\!\!\!-\!\!\!- \,j $ to a disjoint union of an $(m-2)$-component unlink and a Hopf link with components 
numbered $i$ and $j$, all components being zero-framed. If $i=j$, we get an $m$-component unlink where the $i$th component has framing 2 and all other components are zero-framed. 

In fact, it will turn out that $\G_n$ is \emph{isomorphic} to $\cT_n$ when $n$ is even (and to a quotient by certain $2$-torsion in the odd case) via a map that is essentially defined by forgetting the roots of the rooted trees which correspond to iterated commutators determined by the grope branches.

The graded groups $\G_n$ are reminiscent of the groups of knots arising in the theory of finite type invariants, first discovered by Goussarov \cite{Gu1}, defined as a quotient of the monoid of knots by $n$-equivalence. These groups were later constructed more conceptually using claspers by both Habiro and Goussarov \cite{H, Gu2}. For the case of string links with $m$-strands, in \cite{CST} we defined $G_n(m)$ to be the monoid of $m$-string links modulo ($3$-dimensional) grope cobordism of class $n+1$, and the associated graded  quotients are finitely generated abelian groups.
As in the current paper, we defined a surjective realization map $$\Phi\colon \mathcal B^g_n(m)\to G_n(m)/G_{n+1}(m)$$
where $\mathcal B^g_n(m)$ is the group of Feynman diagrams with grope degree $n$, and showed that rationally it is an isomorphism. 
Clearly there is a surjection $\rho\colon G_n(m)/G_{n+1}(m)\twoheadrightarrow {\sf G}_n(m)$, and indeed our realization maps $R_n$ are the compositions $\rho\circ\Phi$, since graphs with loops give rise to null-concordant links and are therefore in the kernel of $\rho\circ\Phi$. (See also Remark~\ref{rem:finite-type-IHX} below.)

Although we could take the composition $\rho\circ\Phi$  as our definition of $R_n$, and prove the required properties
by modifying the $3$-dimensional methods developed for knots by the first and third authors in \cite{CT1,CT2},  we take a different approach here which is directly 4-dimensional and significantly simpler.  

\subsection{The Whitney tower filtration}\label{subsec:W-tower-filtration}
We consider a second filtration on the set of framed links with $m$ components, defined by using (framed) {\em Whitney towers} in place of (framed) gropes. Just like gropes are iterated surface stages glued to each other in a specified way, a Whitney tower is constructed by adding layers of immersed framed Whitney disks that pair up the intersection points of lower stages 
(Figure~\ref{fig:order3Whitneytower-and-withTrees} and Section~\ref{sec:w-towers}). 

Whitney towers also come with a measure of complexity, the {\em order}, which determines how often intersection points are paired up. For example, an order zero Whitney tower is just a union of framed immersed disks in $B^4$ bounded by a link in $S^3$ with the induced framing, and an order~$1$ Whitney tower would also have all intersections among the immersed disks paired by Whitney disks. We define $\bW_n=\bW_n(m)$ to be the set of 
framed $m$-component links that bound a Whitney tower of order~$n$ in $B^4$. On the set of framed links, we thus get the\\
{\em Whitney tower filtration} (by order)
\begin{equation} \tag{$\bW$} 
\dots \subseteq \bW_{3} \subseteq \bW_{2} \subseteq \bW_{1} \subseteq  \bW_{0} \subseteq \bL =\bL(m)
\end{equation}
\begin{figure}[h]
\centerline{\includegraphics[scale=.45]{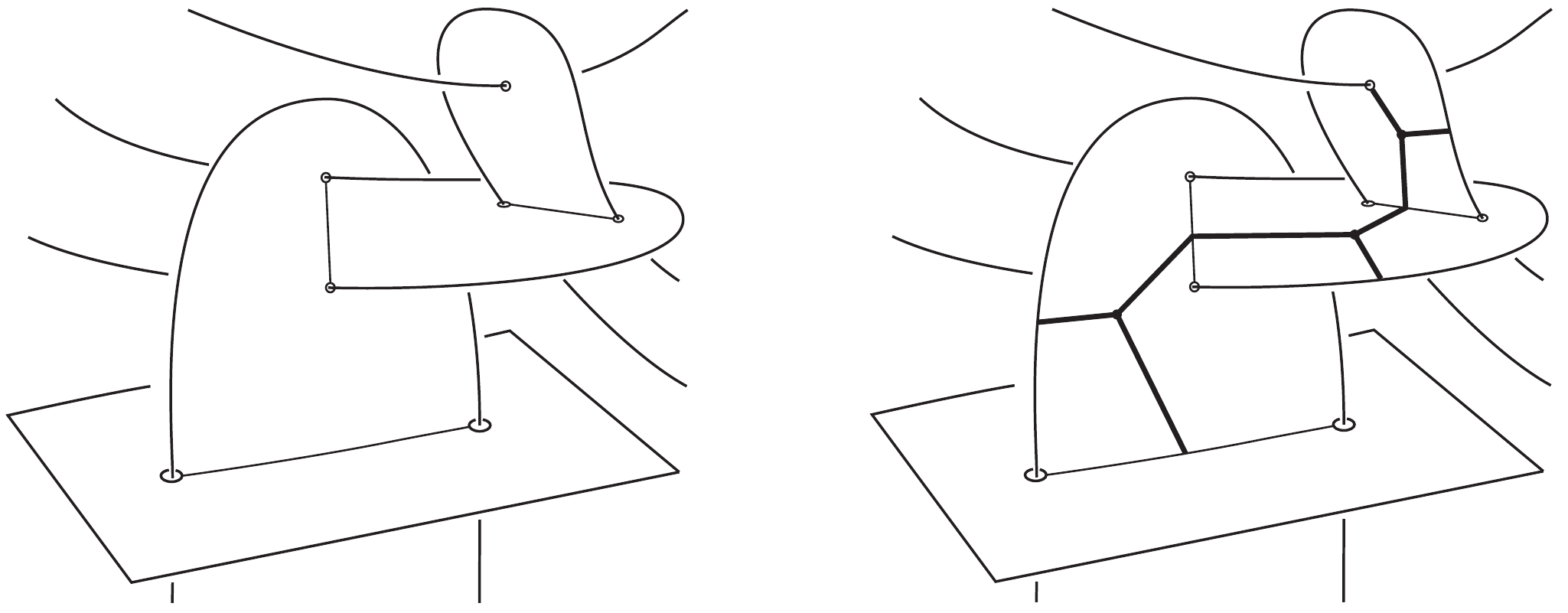}}
         \caption{Part of a Whitney tower in $4$--space (left), and shown with part of an associated tree (right). 
         Whitney disks pair up transverse intersections of opposite sign.}
         \label{fig:order3Whitneytower-and-withTrees}
\end{figure}

As for gropes above, the equivalence relation of order $n+1$ {\em Whitney tower concordance} (see Section~\ref{sec:realization-maps}) then leads to the `associated graded' $\W_n=\W_n(m)$. 
\begin{thm} \label{thm:R-onto-G} 
For each $m,n$, there are surjective realization maps 
\[ 
R_n=R_n(m) : \cT_n(m) \sra\W_n(m)
\] 
and the sets $\W_n(m)$ are finitely generated abelian groups under the well-defined operation of connected sum $\#$. Moreover, $\W_n$ is the set of framed links $L\in\bW_n$ modulo the relation that $[L_1]=[L_2]\in\W_n$ if and only if $L_1 \# -L_2$ lies in $\bW_{n+1}$, for some choice of connected sum $\#$. Here $-L$ is the mirror image of $L$ with reversed framing. 
\end{thm} 

Note that $\bW_0$ consists of links that are evenly framed because a component has even framing if and only if it bounds a framed immersed disk in $B^4$. It is well known that the signed number of self-intersection points of such a framed disk equals twice the framing on the boundary. As a consequence, if a knot bounds the next type of Whitney tower (of order one) then all the intersections of the first stage (order zero) disk are paired up by Whitney disks and so the framing is zero. In fact, $\W_0(1) \cong 2\Z$, given by the framing; and $\W_1(1) \cong\Z_2$, detected by the Arf invariant.

The resemblance between the two realization maps above is no coincidence: It is a result of the second author in \cite{S1} that $\bG_n=\bW_n$ and $\G_n=\W_n$ by completely geometric constructions that lead from gropes to Whitney towers and back (Figure~\ref{fig:wtower-grope}).
\begin{figure}[ht!]
\centerline{\includegraphics[scale=.45]{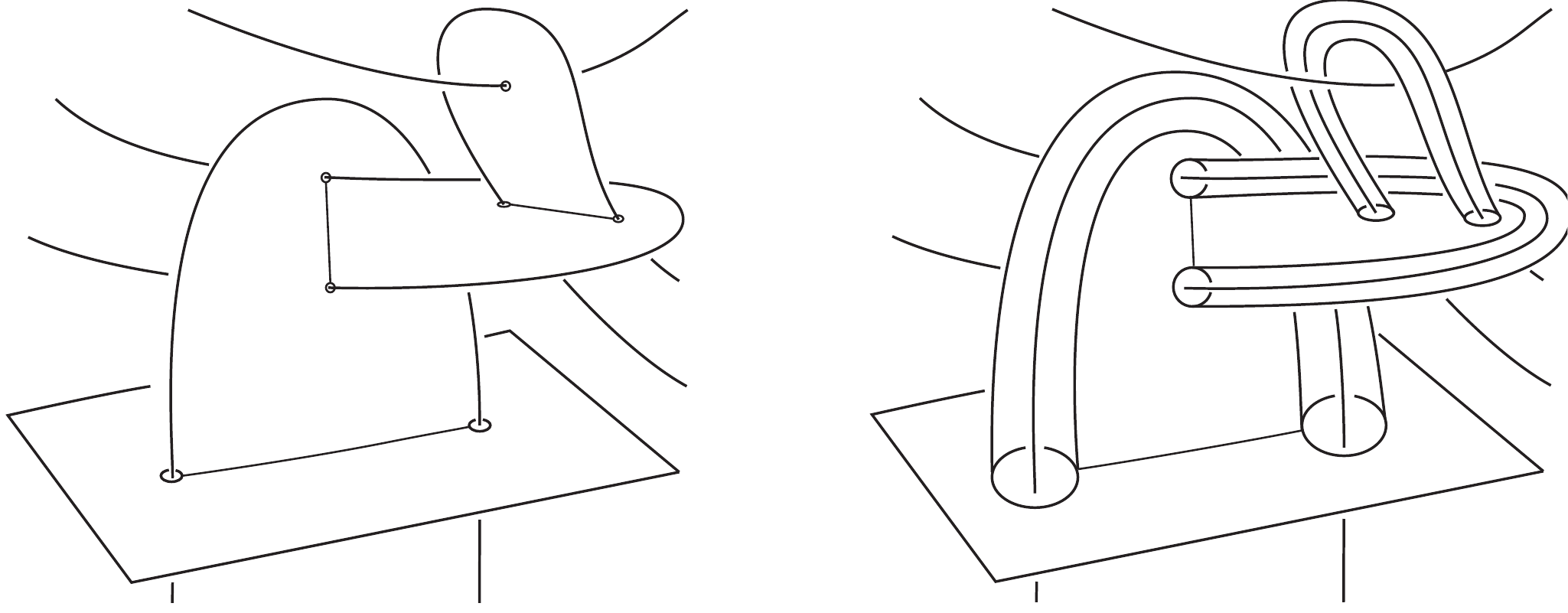}}
         \caption{Controlled conversions between order~$n$ Whitney towers and class~$(n+1)$ gropes are described in detail in \cite{S1}.}
         \label{fig:wtower-grope}
\end{figure}
 This means that we can continue to study either filtration and in this paper, we decide to take the Whitney tower point of view. This choice is motivated by the fact that 
Whitney towers come with an obstruction theory which is accompanied by higher-order intersection invariants taking values in $\cT$. As described in Section~\ref{sec:w-towers} and hinted at in the right side of 
Figure~\ref{fig:order3Whitneytower-and-withTrees}, any Whitney tower $\cW$ of order~$n$ has an 
order $n$ intersection invariant 
$\tau_n(\cW)\in\cT_n$ which is defined by summing the trees pictured in Figure~\ref{fig:order3Whitneytower-and-withTrees}.  Order $n$ intersection invariants vanish on Whitney towers of order $n+1$, and
the key step in the proof of Theorem~\ref{thm:R-onto-G} (given in Section~\ref{sec:realization-maps}) is the following ``raising the order'' result: 

\begin{thm}[\cite{CST,S1,ST2}]\label{thm:raise} 
A link bounds a Whitney tower $\cW$ of order~$n$ with $\tau_n(\cW)=0$, 
if and only if it bounds a Whitney tower of order $n+1$. 
\end{thm} 

This theorem follows from realizing all relations in $\cT_n$ geometrically by controlled maneuvers on Whitney towers. It was a very pleasant surprise that, say, the Jacobi identity (or IHX relation) has an incarnation as a certain sequence of moves on  Whitney disks and their boundary-arcs in a Whitney tower \cite{CST}. It has the following important corollary whose proof is in Section~\ref{def:wtc}.

\begin{cor}\label{cor:tau=w-concordance}
Links $L_0$ and $L_1$ represent the same element of $\W_n$
if and only if there exist order $n$ Whitney towers $\cW_i$ in $B^4$ with $\partial\cW_i=L_i$ and $\tau_n(\cW_0)=\tau_n(\cW_1)\in\cT_n$.
\end{cor}

Our realization map $R_n$ (and its surjectivity) can then be explained extremely well by the following commutative diagram:
\begin{equation}
\xymatrix{
\mathfrak{W}_n\ar@{->>}[d]_{\tau_n}\ar@{->>}[r]^{\partial} & \bW_n \ar@{->>}[d]\\
\cT_n \ar@{->>}[r]^{R_n} &\W_n
}
\end{equation}
An element in the set $\mathfrak{W}_n$ is an order $n$ Whitney tower $\cW$ in $B^4$ whose boundary $\partial \cW$ lies in $\bW_n$ by definition of our filtration. The map $R_n$ then arises directly from above the diagram via Corollary~\ref{cor:tau=w-concordance} and the surjectivity of the intersection invariant $\tau_n$.
 The proof of Theorem~\ref{thm:R-onto-G} will be completed via a Bing doubling (and internal band sum) construction applied to the Hopf link, showing that $\tau_n$ is surjective.

As a consequence of Theorem~\ref{thm:R-onto-G} together with results in \cite{CST2,CST3,CST4}, it turns out that
$\tau_{2n}(\cW)$ only depends on the concordance class of the link $L=\partial\cW$, and not the Whitney tower $\cW$ it bounds; and that the intersection invariants  $\tau_{2n}$ induce inverses to the realization maps $R_{2n}$.
In particular, we have following classification theorem for even orders: 
\begin{thm}[\cite{CST4}] \label{thm:even} 
The maps $R_{2n}:\cT_{2n}\sra \W_{2n}$ are isomorphisms.
\end{thm}

To describe the situation in odd orders, we first introduce the {\em reduced} version $\widetilde\cT_{2n-1}$ of $\cT_{2n-1}$ by dividing out the {\em framing relations}. These relations are the images of homomorphisms 
\[ 
\Delta_{2n-1} : \Z_2 \otimes \cT_{n-1} \to \cT_{2n-1} 
\] 
defined by sending an order $n-1$ tree $t$ to the sum of trees gotten by doubling the subtree adjacent to each univalent vertex of $t$ (see Figure~\ref{fig:Delta trees} and the precise definition in Section~\ref{sec:proof-thm-odd}).   

\begin{thm} \label{thm:odd} 
The odd order realization maps $R_{2n-1}$ vanish on the image of $\Delta_{2n-1}$ and give surjections 
\[ 
\widetilde R_{2n-1} : \widetilde\cT_{2n-1} \sra\W_{2n-1} 
\] 
Moreover, if a link $L$  bounds a Whitney tower $\cW$ of order~$2n-1$ with $\tau_{2n-1}(\cW) \in \im(\Delta_{2n-1})$ then $L$ bounds a Whitney tower of order $2n$. 
\end{thm} 
Here the results of \cite{CST2,CST3,CST4} also apply to the cases where $n$ is \emph{even}, implying that $\widetilde \tau_{4k-1}(\cW) \in \widetilde \cT_{4k-1}$ only depends on the link $L$ and not the Whitney tower $\cW$, and that 
$\widetilde \tau_{4k-1}$ induces an inverse to the realization map $\widetilde R_{4k-1}$.
We get the following classification in these orders:  
\begin{thm} [\cite{CST4}]\label{thm:4k-1} 
 $\widetilde R_{4k-1}:\cT_{4k-1}\sra\widetilde \W_{4k-1}$ are isomorphisms. 
\end{thm} 

The framing relations in odd orders arise from a subtle interplay between
the IHX relations and certain framing obstructions associated to Whitney disks, as will be described below. As suggested by theorems~\ref{thm:odd} and \ref{thm:4k-1},
we believe that these relations do indeed capture all
the relevant geometric aspects regarding the filtration. Setting $\widetilde{\cT}_{2n}:=\cT_{2n}$, we conjecture:    
\begin{conj} \label{conj:4k-3} 
The realization maps give
$\W_n\cong\widetilde \cT_n$ for all $n$.
\end{conj} 
As discussed below and described in detail in \cite{CST2,CST4}, the affirmation of this conjecture is equivalent
to the non-triviality of certain \emph{higher-order Arf invariants} which complete the classification of $\W_{4k-3}$. 

The big picture is best described by
introducing the following generalization of framed Whitney towers.

\subsection{The Twisted Whitney tower filtration}\label{subsec:twisted-W-tower-filtration}
The proof of Theorem~\ref{thm:odd} given in Section~\ref{sec:proof-thm-odd} uses in addition to IHX type maneuvers another well known geometric move on  a Whitney disk, namely the \emph{boundary twist}. A {\em framed} Whitney disk changes into a {\em twisted} Whitney disk by this move, at the cost of creating a new interior intersection in the Whitney disk (Figure~\ref{boundary-twist-and-section-fig}). 

This motivated us to introduce yet another filtration of the set of links  by looking at those links $\bW^\iinfty_n=\bW^\iinfty_n(m)$ with $m$ components that bound {\em twisted} Whitney towers of order $n$ (rather than {\em framed} Whitney towers as in the definition of $\bW_n$). These are Whitney towers, except that certain Whitney disks are allowed to be twisted. We arrive at the\\
{\em Twisted Whitney tower filtration} (by order)
\begin{equation} \tag{$\bW^\iinfty$} 
\dots \subseteq \bW^\iinfty_{3} \subseteq \bW^\iinfty_{2} \subseteq \bW^\iinfty_{1} \subseteq  \bW^\iinfty_{0}=\bL 
\end{equation}
We refer to Section~\ref{sec:realization-maps} for a precise definition, also of the associated graded $\W^\iinfty_n=\W^\iinfty_n(m)$; and to Section~\ref{sec:w-towers} for details on twisted Whitney towers, including the associated \emph{twisted intersection invariant} $\tau_n^\iinfty(\cW)\in\cT^\iinfty_n= \cT^\iinfty_n (m)$. 
The groups  $\cT^\iinfty_{2n-1}$ are defined as quotients of $\cT_{2n-1}$ by the subgroups generated by trees of the form 
\[ 
  i\,-\!\!\!\!\!-\!\!\!<^{\,J}_{\,J}
\] 
where $J$ is a subtree of order~$n-1$. These \emph{boundary-twist relations} correspond to the intersections created by performing a boundary-twist on an order $n$ Whitney disk. The groups $\cT^\iinfty_{2n}$ include certain ``twisted'' \emph{$\iinfty$-trees} representing framing obstructions on order $n$ Whitney disks, which are not required to be framed in an order $2n$ twisted Whitney tower. In \cite{CST4} we show that 
$\cT$ is the universal home for invariant symmetric bilinear forms on quasi-Lie algebras, and that $\cT^\iinfty_{2n}$ is the universal (symmetric quadratic) refinement of this form in order $n$.  

Here and in the following the symbol $\iinfty$ represents a {\em twist}, and in particular does {\em not} stand for ``infinity''. 

\begin{thm} \label{thm:twisted} 
For each $m,n$, there are surjective realization maps 
\[ 
R^\iinfty_n=R^\iinfty_n(m) : \cT^\iinfty_n(m) \sra\W^\iinfty_n(m) 
\] 
and the sets $\W^\iinfty_n(m)$ are finitely generated abelian groups under the well-defined connected sum $\#$ operation. Moreover, $\W^\iinfty_n$ is the set of framed links $L\in\bW^\iinfty_n$ modulo the relation that $[L_1]=[L_2]\in\W^\iinfty_n$ if and only if $L_1 \# -L_2$ lies in $\bW^\iinfty_{n+1}$, for some choice of connected sum $\#$. Here $-L$ is the mirror image of $L$ with reversed framing. 
\end{thm} 
As before, the key step in proving this result is the following criterion for {\em raising the order of a twisted Whitney tower}:
\begin{thm} \label{thm:twisted-raising-intro} 
A link bounds a twisted Whitney tower $\cW$ of order~$n$ with $\tau^{\iinfty}_{n}(\cW)=0$ if and only if it bounds a twisted Whitney tower of order $n+1$. 
\end{thm} 
Theorem~\ref{thm:twisted-raising-intro} follows from the more general Theorem~\ref{thm:twisted-order-raising-on-A} in Section~\ref{sec:w-towers}, which is proved in Section~\ref{sec:proof-twisted-thm} using ``twisted Whitney moves''
(Lemma~\ref{lem:twistedIHX}) as well as boundary-twists and a construction for ``geometrically cancelling'' twisted Whitney disks. 

In the twisted setting we also have classifications of the associated graded groups of links in three out of four cases: 
\begin{thm}[\cite{CST2,CST4}] \label{thm:twisted-isos} 
For $n\equiv 0,1,3\mod 4$, the realization maps $R^\iinfty_{n}:\cT^\iinfty_n \sra \W^\iinfty_n$ are isomorphisms. 	
\end{thm} 
In fact, in these orders the $\W^\iinfty_n$ are isomorphic to image $\sD_n$ of the first nonvanishing order $n$ (length $n+2$) Milnor link invariants \cite{CST2,CST4}. Theorem~\ref{thm:twisted} (as well as theorems~\ref{thm:even} and \ref{thm:4k-1} above) depends in an essential way on our resolution \cite{CST3} of a combinatorial conjecture of J. Levine involving a map from unrooted trees representing
Whitney tower intersections to rooted trees representing iterated commutators determined by link longitudes. As described in \cite{CST2,CST3}, the kernel of this map is generated by certain symmetric $\iinfty$-trees that are not detected by Milnor invariants. These trees are related to what we call {\em higher-order Arf invariants} of the twisted filtration in \cite{CST2}, and the conjectured non-triviality of these invariants is shown in \cite{CST4} to be equivalent to the above Conjecture~\ref{conj:4k-3} as well as the following:
\begin{conj} \label{conj:twisted} 
The maps $R^\iinfty_{4k-2}:\cT^\iinfty_{4k-2} \sra \W^\iinfty_{4k-2}$ are isomorphisms. 
\end{conj} 

\subsection{Framed versus twisted Whitney towers}\label{subsec:overview}
There is a surprisingly simple relation between twisted and framed Whitney towers of various orders that's very well expressed in terms of the following result:

\begin{prop}\label{prop:exact sequence}
For any $n\in\N$, there is an exact sequence
\[
\xymatrix{ 
0\ar[r] & \W_{2n} \ar[r] &  \W^\iinfty_{2n} \ar[r] & \W_{2n-1}  \ar[r] &  \W^\iinfty_{2n-1} \ar[r] & 0
 } \]
where all maps are induced by the identity on the set of links. 
\end{prop}

To see why this is true, observe that it is easier to find a twisted Whitney tower than a framed one and hence there is a natural inclusion $\bW_{n} \subseteq \bW_{n}^\iinfty$ between our two filtrations $(\bW)$ and $(\bW^\iinfty)$ on the set of links. Moreover, by Definition~\ref{def:twisted-W-towers} we actually have 
\[
\bW_{2n-1}^\iinfty = \bW_{2n-1}
\]
because we found that a twisted Whitney disk of order~$n$ is most naturally associated to an order $2n$ Whitney tower, compare Remark~\ref{rem:motivate-twisted-towers}.  One then needs to show that indeed 
$\bW_{2n}^\iinfty \subseteq \bW_{2n-1}$, which is accomplished in Lemma~\ref{lem:boundary-twisted-IHX} of Section~\ref{sec:proof-twisted-thm} using boundary-twists and twisted Whitney moves. Finally, by the `Moreover' parts of Theorems \ref{thm:R-onto-G} and \ref{thm:twisted}, we may say that 
\[
\W_n = \bW_n/ \bW_{n+1}  \quad \text{ and } \quad \W^\iinfty_n = \bW^\iinfty_n/  \bW^\iinfty_{n+1} 
\]
which implies the exact sequence in Proposition~\ref{prop:exact sequence}. 

If our above conjectures above hold, then for every $n$ the various realization maps should lead to the analogous exact sequence for our groups defined by trees. The following result shows that this is indeed the case and even gives more information on kernels (respectively cokernels). It is best expressed in terms of the \emph{free quasi-Lie algebra} on $m$ generators:
\[
\sL'=\sL'(m)= \bigoplus_{n\in\N} \sL'_n
\]
\begin{defn}[compare \cite{L3}]\label{def:L'}
The abelian group $\sL'_{n+1}=\sL'_{n+1}(m)$ is generated by order~$n$ (trivalent, oriented) \emph{rooted} trees, each having a specified univalent vertex, labeled as {\em root}, and all other univalent vertices labelled by elements of $\{1,\dots,m\}$, modulo the AS and IHX relations of Figure~\ref{fig:ASandIHXtree-relations}.
\end{defn}
Here the prefix `quasi' reflects the fact that, although the IHX relation corresponds
to the Jacobi identity via the usual identification of rooted trees with non-associative brackets, the usual Lie algebra self-annihilation relation $[X,X]=0$ does 
{\em not} hold in $\sL'$. 
It is replaced by the weaker anti-symmetry (AS) relation $[Y,X]=-[X,Y]$.

\begin{thm}[\cite{CST4}]\label{thm:Tau-sequences}
For any $m,n\in\N$, there are short exact sequences
\[
\xymatrix{ 
0\ar[r] & \cT_{2n} \ar[r] &  \cT^\iinfty_{2n} \ar[r] & \Z_2 \otimes \sL'_{n+1} \ar[r] & 0
 } \]
and
\[
\xymatrix{ 
0\ar[r] &  \Z_2 \otimes \sL'_{n+1} \ar[r] & \widetilde\cT_{2n-1}  \ar[r] &  \cT^\iinfty_{2n-1} \ar[r] & 0
 } \]
\end{thm}
This result is proved in \cite{CST4} using the \emph{universality} of $ \cT^\iinfty_{2n}$ as the target of quadratic refinements of the canonical `inner product' pairing 
\[
\langle \ , \ \rangle: \sL'_{n+1} \times \sL'_{n+1} \to \cT_{2n}
\]
given by gluing the roots of two rooted trees, see Definition~\ref{def:trees}.
We can connect these short exact sequences to a single exact sequence
\[
\xymatrix{ 
0\ar[r] & \cT_{2n} \ar[r] &  \cT^\iinfty_{2n} \ar[r] & \widetilde\cT_{2n-1}  \ar[r] &  \cT^\iinfty_{2n-1} \ar[r] & 0
 } \]
just like in  Proposition~\ref{prop:exact sequence}. However, Theorem~\ref{thm:Tau-sequences} makes the additional predictions that
\[
\Cok ( \W_{2n} \to  \W^\iinfty_{2n}) \cong \Z_2 \otimes \sL'_{n+1} \cong \Ker(  \W_{2n-1}  \to  \W^\iinfty_{2n-1} )
\]
assuming the conjectures in this paper. As a consequence of these conjectures, we would obtain new concordance invariants with values in $\Z_2 \otimes \sL'_{n+1}$ and defined on $\bW^\iinfty_{2n}$, as the obstructions for a link to bound a framed Whitney tower of order~$2n$. In \cite{CST4} we show that a quotient of $\Z_2 \otimes \sL'_{n+1}$, namely $\Z_2 \otimes \sL_{n+1}$,  is indeed detected by what we call {\em higher-order Sato-Levine invariants}. Here the free quasi-Lie algebra $\sL'$ is 
replaced by the usual \emph{free Lie algebra} 
$\sL$ satisfying the Jacobi identity and self-annihilation relations $[X,X]=0$. 

Levine showed in \cite{L2} that for odd $n$ the squaring map $X\mapsto [X,X]$ induces an isomorphism
$$
\Z_2 \otimes \sL_{\frac{n+1}{2}}
\cong\Ker(\Z_2 \otimes \sL'_{n+1}\twoheadrightarrow\Z_2 \otimes \sL_{n+1}) 
$$
We thus conjecture that this group is also detected by concordance invariants which are the above-mentioned {\em higher-order Arf invariants}. These would generalize the usual Arf invariants of the link components, which (as shown in \cite{CST2}) is the case $n=1$.

It is interesting to note that the case $n=0$ leads to the prediction
$$
\Cok ( \W_{0} \to  \W^\iinfty_{0}) \cong \Z_2 \otimes \sL_{1} \cong (\Z_2)^m
$$
This is indeed the group of framed $m$-component links modulo those with even framings! In fact, the consistency of this computation was the motivating factor to consider filtrations of the set of {\em framed} links $\bL$, rather than just oriented links.

As described in \cite{CST4}, the higher-order Sato-Levine invariants turn out to be determined by the Milnor invariants, providing a new interpretation for these classical invariants as obstructions to ``untwisting'' a Whitney tower. On the other hand, the higher-order Arf invariants do not appear to correspond to any known invariants. They take values in a (currently unknown) quotient of $\Z_2 \otimes \sL_{\frac{n+1}{2}}$ and we show in \cite{CST4} that together with the Milnor invariants they give a complete characterization of our three filtrations.

\subsection{Summary of all papers in this series}
Paper \cite{CST0} gives an overview of the results of this paper together with the closely related papers \cite{CST2,CST3,CST4,CST5}. The classifications of the geometric filtrations of link concordance defined in the current paper are achieved in a sequence of steps: 
\begin{enumerate}
\item The current paper extends our previous intersection theory of Whitney towers \cite{CST,ST2} to the reduced (and twisted) settings, and proves that our obstruction theory works: If the order $n$ intersection invariant of a (twisted) Whitney tower $\cW$ vanishes in $\widetilde{\cT}$
  (resp. $\cT^\iinfty$) then $\cW$ ``can be raised'' from order $n$ to order $n+1$, without changing the link $\partial \cW$. As a consequence, we obtain the realization epimorphisms: 
\[ 
\widetilde R_n: \widetilde\cT_n \sra\W_n \quad \mbox{and} \quad  R^\iinfty_{n}:\cT^\iinfty_n \sra \W^\iinfty_n
\] 
We also introduce the exact sequences which explain the relationship between
framed and twisted Whitney towers, motivating the definitions of the higher-order
Sato-Levine and higher-order Arf invariants.

\item In \cite{CST2} we first show that the order $n$ Milnor invariants for links can be thought of as an epimorphism
\[
\mu_n:\W^\iinfty_n \sra \sD_n:=\ker([\cdot,\cdot]:\sL_1\otimes \sL_{n+1}\to \sL_{n+2})
\]
Note that $\sD_n$ is a free abelian group whose rank can be computed via the Hall basis algorithm. Secondly, we use the geometric notion of {\em grope duality} to show that the composition
\[
\eta_n: \cT^\iinfty_n \overset{R^\iinfty_n}{\sra} \W^\iinfty_n \overset{\mu_n}{\sra} \sD_n
\]
can be described combinatorially in a very simple way: it is given by summing over all ways of choosing a root on a given tree. This map $\eta_n$ is closely related
to the procedure for converting a Whitney tower into a grope which gives the isomorphism $\W_n\cong\G_n$.
The construction of boundary links realizing the image of the higher-order Arf invariants leads to 
new geometric characterizations of links with vanishing Milnor invariants through
length $2n$.

\item In \cite{CST3} we use combinatorial Morse theory to prove the \emph{Levine Conjecture}, that a map $\eta_n'$ (analogous to $\eta_n$), which is also defined by summing over all 
root choices, gives an isomorphism $\cT_n \cong \sD'_n$,
where $\sD'_n$ is the bracket kernel on the free quasi-Lie algebra $\sL'$ (analogous to $\sD_n$). This result has implications in the study of $3$-dimensional homology cylinders, where Levine originally formulated his conjecture, as well as playing a key role in providing the algebraic framework for the classifications of $\W_n$ and $\W^\iinfty_n$ completed in \cite{CST4}.

\item In \cite{CST4}, we assemble the relevant groups into commutative diagrams of exact sequences and complete the classifications of the geometric filtrations of link concordance defined in the current paper. For instance, our resolution in \cite{CST3} of the Levine Conjecture is used to show that $\eta_n$ is an isomorphism for $n\equiv 0,1,3\mod 4$ and that its kernel is $\Z_2\otimes \sL_k$ for $n=4k-2$. 
This leads to the formulations of the higher-order Arf invariants in both the framed and twisted settings, as well as a demonstration of their equivalence, and their relation to the higher-order Sato-Levine invariants. The geometric filtrations are shown to be classified by the Milnor and higher-order Arf invariants.
This allows us to give complete proofs of Theorems~\ref{thm:even}, \ref{thm:4k-1} and \ref{thm:twisted} above.
 Moreover, we show in what sense the twisted intersection invariant $\tau_n^\iinfty$ is the universal quadratic refinement of its framed partner $\tau_n$, and complete the algebraic description of the relationship between the framed and twisted Whitney tower filtrations (Theorem~\ref{thm:Tau-sequences} above).
\item Further applications to filtrations of string links and homology cylinders are described in \cite{CST5}. 
\end{enumerate}

We emphasize that although the Milnor and higher-order Arf invariants completely classify the groups $\W_n$ and $\W^\iinfty_n$, there remains the question of 
determining the exact (finite) 2-torsion group which is the range of the higher-order Arf invariants, as touched on above and elaborated on throughout these papers.



{\bf Acknowledgments:} This paper was partially written while the first two authors were visiting the third author at the Max-Planck-Institut f\"ur Mathematik in Bonn. They all thank MPIM for its stimulating research environment and generous support. The exposition of this paper was significantly improved by a careful and insightful anonymous referee. The first author was also supported by NSF grant DMS-0604351 and the last author was also supported by NSF grants DMS-0806052 and DMS-0757312. The second author was partially supported by PSC-CUNY research grant PSCREG-41-386. 

\section{Whitney towers}\label{sec:w-towers}
We sketch here the relevant theory of Whitney towers as developed in \cite{CST,S1,ST2}, giving details for the new notion of \emph{twisted} Whitney towers. We work in the \emph{smooth oriented} category (with orientations usually suppressed from notation), even though all our results hold in the locally flat topological category by the basic results on topological immersions in Freedman--Quinn \cite{FQ}. In fact, it can be shown that the filtrations  $\mathbb G_n$, $\mathbb W_n$ and $\mathbb W^\iinfty_n$ are identical in the smooth and locally flat settings. This is because a topologically flat surface can be promoted to a smooth surface at the cost of only creating unpaired intersections of arbitrarily high order (see Remark~\ref{rem:locally-flat-and-smooth}).

\subsection{Operations on trees}\label{subsec:trees}
To describe Whitney towers it is convenient to use the bijective correspondence
between formal non-associative bracketings of elements from
the index set $\{1,2,3,\ldots,m\}$ and
rooted trees, trivalent and oriented as in Definition~\ref{def:Tau},
with each univalent vertex labeled by an element from the index set, except
for the \emph{root} univalent vertex which is left unlabeled. 

\begin{defn}\label{def:trees}
Let $I$ and $J$ be two rooted trees.
\begin{enumerate} 
\item The \emph{rooted product} $(I,J)$ is the rooted tree gotten
by identifying the root vertices of $I$ and $J$ to a single vertex $v$ and sprouting a new rooted edge at $v$.
This operation corresponds to the formal bracket (Figure~\ref{inner-product-trees-fig} upper right). The orientation of $(I,J)$ is inherited from those of $I$ and $J$ as well as the order in which they are glued.

\item The \emph{inner product}  $\langle I,J \rangle $ is the
unrooted tree gotten by identifying the roots of $I$ and $J$ to a single non-vertex point.
Note that $\langle I,J \rangle $ inherits an orientation from $I$ and $J$, and that
all the univalent vertices of $\langle I,J \rangle $ are labeled.
(Figure~\ref{inner-product-trees-fig} lower right.)

\item The \emph{order} of a tree, rooted or unrooted, is defined to be the number of trivalent vertices.
\end{enumerate}
\end{defn}
The notation of this paper will not distinguish between a bracketing and its corresponding rooted tree
(as opposed to the notation $I$ and $t(I)$ used in \cite{S1,ST2}).
In \cite{S1,ST2} the inner product is written as a dot-product, and the rooted product
is denoted by $*$.

\begin{figure}[ht!]
        \centerline{\includegraphics[scale=.40]{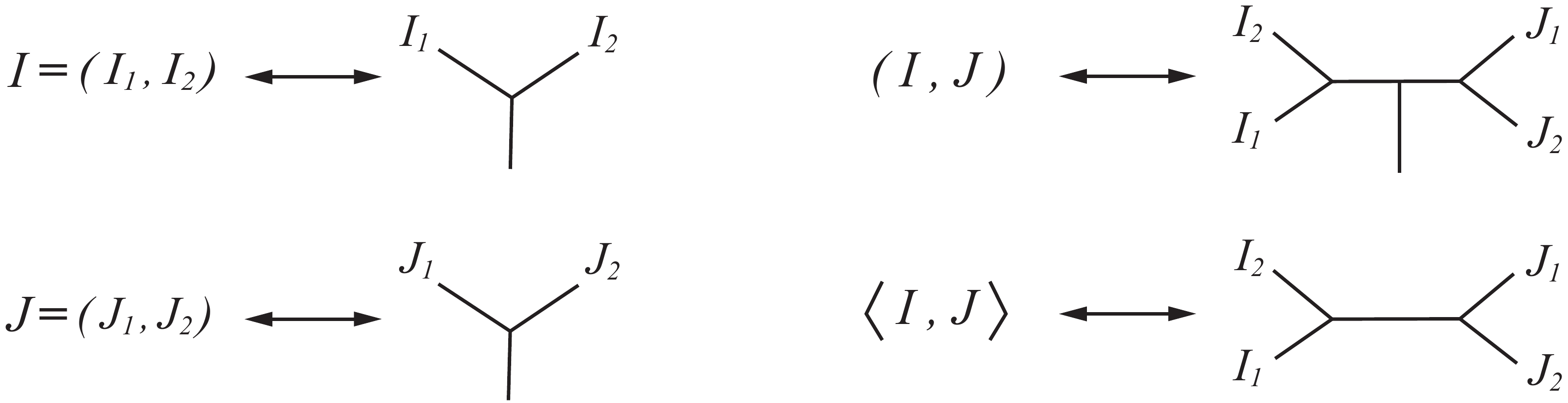}}
        \caption{The \emph{rooted product} $(I,J)$ and \emph{inner product} $\langle I,J \rangle$ of $I=(I_1,I_2)$ and $J=(J_1,J_2)$. All trivalent orientations correspond to a clockwise orientation of the plane.}
        \label{inner-product-trees-fig}
\end{figure}

\subsection{Whitney disks and higher-order intersections}\label{subsec:order-zero-w-towers-and-ints}
A collection $A_1,\ldots,A_m\looparrowright (M,\partial M)$ of 
connected surfaces in a $4$--manifold $M$ is a \emph{Whitney tower of order zero} if the $A_i$ are \emph{properly immersed} in the sense that the boundary is embedded in $\partial M$ and the interior is generically immersed in $M \smallsetminus \partial M$. 




To each order zero surface $A_i$ is associated
the order zero rooted tree consisting of an edge with one vertex labeled by $i$, and
to each transverse intersection $p\in A_i\cap A_j$ is associated the order zero
tree $t_p:=\langle i,j \rangle$ consisting of an edge with vertices labelled by $i$ and $j$. Note that
for singleton brackets (rooted edges) we drop the bracket from notation, writing $i$ for $(i)$.

The order 1 rooted Y-tree $(i,j)$, with a single trivalent vertex and two univalent labels $i$ and $j$,
is associated to any Whitney disk $W_{(i,j)}$ pairing intersections between $A_i$ and $A_j$. This rooted tree
can be thought of as being embedded in $M$, with its trivalent vertex and rooted
edge sitting in $W_{(i,j)}$, and its two other edges descending into $A_i$ and $A_j$ as sheet-changing paths. (The cyclic orientation at the trivalent vertex of the bracket  $(i,j)$  corresponds to an orientation of $W_{(i,j)}$ via a convention described below in \ref{subsec:w-tower-orientations}.)

Recursively, the rooted tree $(I,J)$ is associated to any Whitney disk $W_{(I,J)}$ pairing intersections
between $W_I$ and $W_J$ (see left-hand side of Figure~\ref{WdiskIJandIJKint-fig}); with the understanding that if, say, $I$ is just a singleton $i$, then $W_I$ denotes the order zero surface $A_i$. Note that a $W_{(I,J)}$ can be created by a finger move pushing $W_J$ through $W_I$.

To any transverse intersection $p\in W_{(I,J)}\cap W_K$ between $W_{(I,J)}$ and any
$W_K$ is associated the un-rooted tree $t_p:=\langle (I,J),K \rangle$  (see right-hand side of Figure~\ref{WdiskIJandIJKint-fig}).

\begin{figure}[ht!]
        \centerline{\includegraphics[width=120mm]{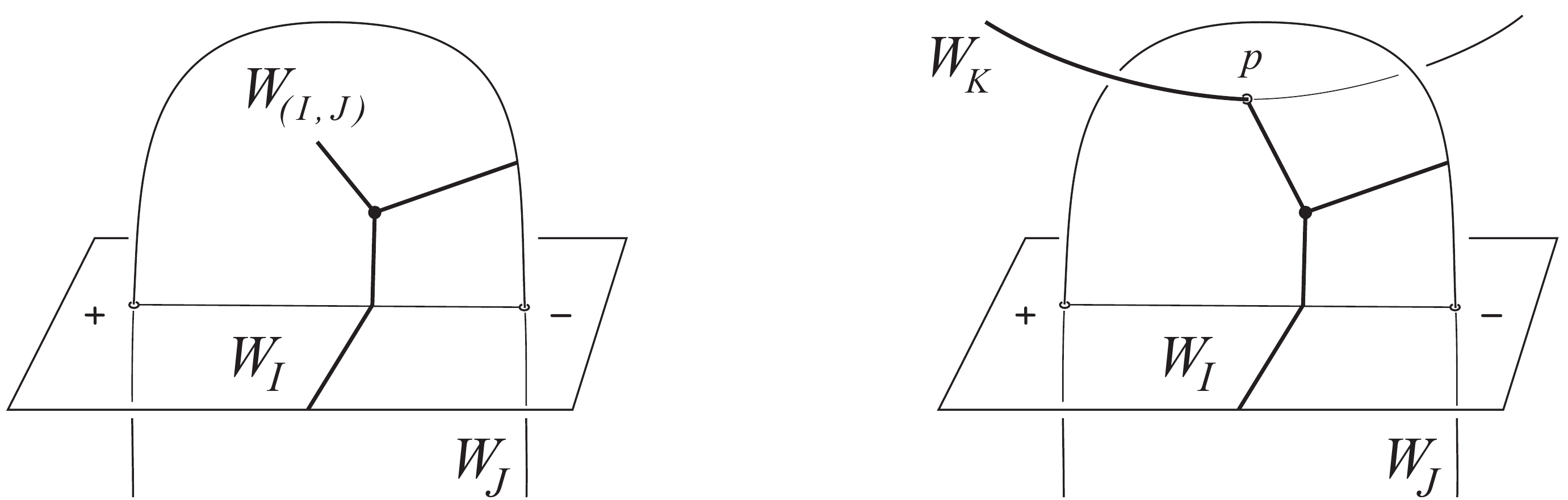}}
        \caption{On the left, (part of) the rooted tree $(I,J)$ associated to a Whitney disk $W_{(I,J)}$. On the right, (part of) the unrooted tree $t_p=\langle (I,J),K \rangle$ associated to an intersection $p\in W_{(I,J)}\cap W_K$. Note that $p$ corresponds to where the roots of $(I,J)$ and $K$ are identified to a (non-vertex) point in $\langle (I,J),K \rangle$.}
        \label{WdiskIJandIJKint-fig}
\end{figure}

\begin{defn}\label{def:int-and-Wdisk-order}
The \emph{order of a Whitney disk} $W_I$ is defined to be the order of the rooted tree $I$, and the \emph{order of a transverse intersection} $p$ is defined to be the order of the tree $t_p$.
\end{defn}

\begin{defn}\label{def:framed-tower}
A collection $\cW$ of properly immersed surfaces together with higher-order
Whitney disks is an \emph{order $n$ Whitney tower}
if $\cW$ contains no unpaired intersections of order less than $n$.  
\end{defn}
The Whitney disks in $\cW$ must have disjointly embedded boundaries, and generically immersed interiors.  All Whitney disks and order zero surfaces must also be \emph{framed} (as discussed next).

\subsection{Twisted Whitney disks and framings}\label{subsec:twisted-w-disks}
The normal disk-bundle of a Whitney disk $W$ in $M$ is isomorphic to $D^2\times D^2$,
and comes equipped with a canonical nowhere-vanishing \emph{Whitney section} over the boundary given by pushing $\partial W$  tangentially along one sheet and normally along the other, avoiding the tangential direction of $W$
(see Figure~\ref{Framing-of-Wdisk-fig}, and e.g.~1.7 of \cite{Sc}).
Pulling back the orientation of $M$ with the requirement that the normal disks
have $+1$ intersection with $W$ means the Whitney section determines
a well-defined (independent of the orientation of $W$)
relative Euler number $\omega(W)\in\Z$ which represents the obstruction to extending
the Whitney section across $W$. Following traditional terminology, when $\omega(W)$ vanishes $W$ is said to be \emph{framed}. (Since $D^2\times D^2$ has a unique trivialization up to homotopy, this terminology is only mildly abusive.)
In general when $\omega(W)=k$, we say that $W$ is
$k$-\emph{twisted}, or just \emph{twisted} if the value of $\omega(W)$ is not specified.
So a $0$-twisted Whitney disks is a framed Whitney disk.

\begin{figure}[ht!]
        \centerline{\includegraphics[width=120mm]{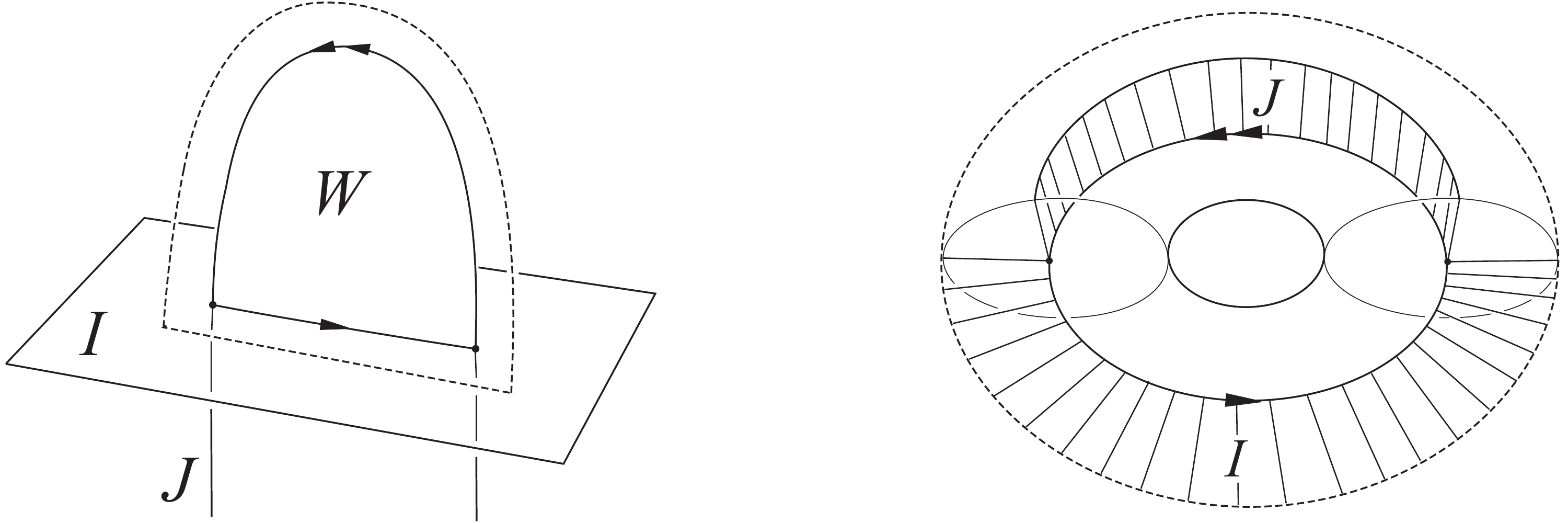}}
        \caption{The Whitney section over the boundary of a framed Whitney disk is
         indicated by the dotted loop shown on the left for a clean Whitney disk $W$ in
         a 3-dimensional slice of 4-space. On the right is shown an embedding into $3$--space of the normal
         disk-bundle over $\partial W$, indicating how the Whitney section determines a well-defined nowhere 		
         vanishing section which lies in the $I$-sheet and is normal to the $J$-sheet. }
        \label{Framing-of-Wdisk-fig}
\end{figure}

Note that a {\em framing} of $\partial A_i$ (respectively $A_i$) is by definition a trivialization of the normal bundle of the immersion. If the ambient $4$-manifold is oriented, this is equivalent to an orientation and a nonvanishing normal vector field on $\partial A_i$ (respectively $A_i$).
The twisting $\omega(A_i)\in\Z$ of an order zero surface is also defined when a framing of $\partial A_i$ is given, and $A_i$ is said to be \emph{framed} 
when $\omega(A_i)=0$.


\subsection{Twisted Whitney towers}\label{subsec:intro-twisted-w-towers}
In the definition of an order $n$ Whitney tower given just above (following \cite{CST,S1,S2,ST2})
all Whitney disks and order zero surfaces are required to be framed. It turns out that the natural generalization to twisted Whitney towers involves allowing twisted Whitney disks only in at least ``half the order'' as follows:

\begin{defn}\label{def:twisted-W-towers}
A \emph{twisted Whitney tower of order $0$} is a collection of properly immersed surfaces
in a $4$--manifold
(without any framing requirement).

For $n>0$, a \emph{twisted Whitney tower of order $(2n-1)$} is just a (framed) Whitney
tower of order $(2n-1)$ as in Definition~\ref{def:framed-tower} above.

For $n>0$, a \emph{twisted Whitney tower of order $2n$} is a Whitney
tower having all intersections of order less than $2n$ paired by
Whitney disks, with all Whitney disks of order less than $n$ required to be framed, but Whitney disks of order at least $n$ allowed to be twisted.
\end{defn}

\begin{rem}\label{rem:framed-is-twisted}
Note that, for any $n$, an order $n$ (framed) Whitney tower
is also an order $n$ twisted Whitney tower. We may
sometimes refer to a Whitney tower as a \emph{framed} Whitney tower to emphasize
the distinction, and will always use the adjective ``twisted'' in the setting of
Definition~\ref{def:twisted-W-towers}.
\end{rem}

\begin{rem}\label{rem:motivate-twisted-towers}
The convention of allowing only order $\geq n$ twisted Whitney disks in order $2n$ twisted Whitney towers is explained both algebraically and geometrically in \cite{CST2}. In any event, an order $2n$ twisted Whitney tower can always be modified
so that all its Whitney disks of order $>n$ are framed, so the twisted Whitney disks of order equal to $n$ are the important ones.  
\end{rem}


 \subsection{Whitney tower orientations}\label{subsec:w-tower-orientations}

Orientations on order zero surfaces in a Whitney tower $\cW$ are fixed, and required to induce the orientations
on their boundaries.
After choosing and fixing orientations on all the Whitney disks in
$\cW$, the associated trees 
are embedded in $\cW$ so that the vertex orientations are induced from
the Whitney disk orientations, with the descending edges of each
trivalent vertex enclosing the \emph{negative intersection point} of the corresponding Whitney disk, as in Figure~\ref{WdiskIJandIJKint-fig}.
(In fact, if a tree $t$ has more than one trivalent vertex which corresponds to the same Whitney disk, then
$t$ will only be immersed in $\cW$, but this immersion can be taken to be a local embedding around each trivalent vertex of $t$
as in Figure~\ref{WdiskIJandIJKint-fig}.)

This ``negative corner'' convention, which differs from the
positive corner convention in \cite{CST,ST2}, will turn out to be
compatible with commutator conventions for use in \cite{CST2}.

With these conventions, different choices of orientations on Whitney disks in $\cW$ correspond to anti-symmetry relations (as explained in \cite{ST2}).

  
\subsection{Intersection invariants for Whitney towers}\label{subsec:intro-w-tower-int-invariants}
We recall from Definition~\ref{def:Tau} that the abelian group $\cT_n$ is the free abelian group on labeled
vertex-oriented order $n$ trees, modulo the AS and IHX relations, see Figure~\ref{fig:ASandIHXtree-relations}.
The obstruction theory of \cite{ST2} in the current simply connected setting works as follows.
\begin{defn}
The \emph{order $n$ intersection invariant} $\tau_n(\cW)$ of an order
$n$ Whitney tower $\cW$ is defined to be
$$
\tau_n(\cW):=\sum \epsilon_p\cdot t_p \in\cT_n
$$ 
where the sum is over all order $n$ intersections $p$,
with
$\epsilon_p=\pm 1$ the usual sign of a transverse intersection
point.
\end{defn}

As stated in Theorem~\ref{thm:raise} in the introduction, 
if $L$ bounds $\cW\subset B^4$ with $\tau_n(\cW)=0\in \cT_n$, then $L$ bounds a Whitney tower
of order $n+1$. This is a special case of the simply connected version of the more general Theorem~2 of \cite{ST2}. 
We will use the following version of Theorem~2 of \cite{ST2}
where the order zero surfaces are either properly immersed disks in 
$B^4$ or properly immersed annuli in $S^3\times I$:  

\begin{thm}[\cite{ST2}]\label{thm:framed-order-raising-on-A}
If a collection $A$ of properly immersed surfaces in a simply connected $4$--manifold supports an order $n$ Whitney tower $\cW$ with $\tau_n(\cW)=0\in\cT_n$, then $A$ is regularly homotopic (rel $\partial$) to 
$A'$ which supports an order $n+1$ Whitney tower.
\end{thm}


\subsection{Intersection invariants for twisted Whitney towers}\label{subsec:intro-twisted-w-tower-int-invariants}

The intersection invariants for Whitney towers are extended to
twisted Whitney towers as follows:

\begin{defn}\label{def:T-infty-odd}
The abelian group $\cT^{\iinfty}_{2n-1}$ is the quotient of $\cT_{2n-1}$ by the \emph{boundary-twist relations}: 
\[
\langle  (i,J),J \rangle \,=\, i\,-\!\!\!\!\!-\!\!\!<^{\,J}_{\,J}\,\,=\,0
\] 
Here $J$ ranges over all order $n-1$ rooted trees.
(This is the same as taking the quotient of $\widetilde{\cT}_{2n-1}=\cT_{2n-1}/\im(\Delta_{2n-1})$ by boundary-twist relations since $\im(\Delta_{2n-1})$ is contained in the span of boundary-twist relations -- see Section~\ref{sec:proof-thm-odd}).
\end{defn}

The boundary-twist relations
correspond geometrically to the fact that 
performing a boundary twist (Figure~\ref{boundary-twist-and-section-fig}) on an order $n$ Whitney disk $W_{(i,J)}$ creates an order $2n-1$ intersection point
$p\in W_{(i,J)}\cap W_J$ with associated tree $t_p=\langle  (i,J),J \rangle $ (which is 2-torsion
by the AS relations) and changes $\omega (W_{(i,J))})$ by $\pm1$. Since order $n$ twisted Whitney disks are allowed in an order $2n$ Whitney tower such trees do not represent obstructions to the existence of the next order twisted tower.

For any rooted tree $J$ we define the corresponding {\em $\iinfty$-tree}, denoted by $J^\iinfty$, by labeling the root univalent vertex with the symbol ``$\iinfty$'':
$$
J^\iinfty := \iinfty\!-\!\!\!- J 
$$ 


\begin{defn}\label{def:T-infty-even}
The abelian group $\cT^{\iinfty}_{2n}$ is the free abelian group on order $2n$
trees and order $n$ $\iinfty$-trees, modulo the
following relations:
\begin{enumerate}
     \item AS and IHX relations on order $2n$ trees (Figure~\ref{fig:ASandIHXtree-relations})
   \item \emph{symmetry} relations: $(-J)^\iinfty = J^\iinfty$
  \item \emph{twisted IHX} relations: $I^\iinfty=H^\iinfty+X^\iinfty- \langle H,X\rangle $
   \item {\em interior twist} relations: $2\cdot J^\iinfty=\langle J,J\rangle $
\end{enumerate}
\end{defn}

Here the AS and IHX relations are as usual, but they only apply to non-$\iinfty$ trees. 
The \emph{symmetry relation} corresponds to the fact that the relative 
Euler number $\omega(W)$ is independent of the orientation of the Whitney disk $W$, with $-J$ denoting the ``opposite'' orientation of $J$ (meaning that the trivalent orientations differ at an odd number of vertices).
The \emph{twisted IHX relation} corresponds to the effect of performing a Whitney move in the presence of a twisted Whitney disk, as described below in
Lemma~\ref{lem:twistedIHX}. The \emph{interior-twist relation} corresponds to the fact that creating a 
$\pm1$ self-intersection
in a $W_J$ changes the twisting by $\mp 2$ (Figure~\ref{InteriorTwistPositiveEqualsNegative-fig}).
\vspace{.25 in}

\begin{rem}\label{rem:quadratic-form}
The symmetry, twisted IHX, and interior twist relations in $\mathcal T^\iinfty_{2n}$ have a surprisingly natural algebraic interpretation that we explain in \cite{CST4}. The idea is to extend the map $J\mapsto J^\iinfty$ to a {\em symmetric quadratic refinement} $q$ of the bilinear form $\langle \cdot,\cdot\rangle$ on the free quasi-Lie algebra of rooted trees (the intersection form on Whitney disks) by defining $q(J)=J^\iinfty$ and extending to linear combinations by the formula
\[
q(J+K):=J^\iinfty+K^\iinfty+\langle J,K\rangle
\]
Expanding $q(I-H+X)=0$ leads to the 6-term IHX relation
\[
I^\iinfty+H^\iinfty+X^\iinfty=\langle I,H \rangle-\langle I,X \rangle+\langle H,X \rangle
\]
which is equivalent to the twisted IHX relation in the presence of the interior-twist relations. Those in turn follow by setting $K:=-J$ from the symmetry relation.
\end{rem}

\begin{rem}\label{rem:finite-type-IHX}

We discovered in \cite{CST} that the (framed) IHX relation can be realized in three dimensions as well as four, and
it is interesting to note that many of the relations that we obtain for twisted Whitney towers in four dimensions can also be realized by rooted clasper surgeries (grope cobordisms) in three dimensions. Here the twisted Whitney disk corresponds to a $\pm1$ framed leaf of  a clasper.   For example the relation $I^\iinfty=H^\iinfty+X^\iinfty-\langle H,X\rangle$ has the following clasper explanation. $I^\iinfty$ represents a clasper with one isolated twisted leaf. By the topological IHX relation, one can replace $I^\iinfty$ by two claspers of the form $H^\iinfty$ and $(-X)^\iinfty=X^\iinfty$ embedded in a regular neighborhood of the original clasper with leaves parallel to the leaves of the original. The twisted leaves are now linked together, so applying Habiro's zip construction (which complicates the picture considerably) one gets three tree claspers, of the form 
$H^\iinfty$, $X^\iinfty$ and $\langle H,-X\rangle$ respectively. 

Similarly, the relation $2\cdot J^\iinfty=\langle J,J\rangle$ has an interpretation where one takes a clasper which represents $J^\iinfty$ and splits off a geometrically cancelling parallel copy, representing the tree $J^\iinfty$. Again, because the twisted leaves link, we also get the term $\langle J,-J\rangle.$ 

These observations will be enlarged upon in \cite{CST5} to analyze filtrations on homology cylinders.
\end{rem}

Recall from Definition~\ref{def:twisted-W-towers} (and Remark~\ref{rem:motivate-twisted-towers}) that twisted Whitney disks
only occur in even order twisted Whitney towers, and only those of half-order are
relevant to the obstruction theory. 
\begin{defn}\label{def:tau-infty}
The \emph{order $n$ intersection intersection invariant}
$\tau_{n}^{\iinfty}(\cW)$ of an order
$n$ twisted Whitney tower $\cW$ is defined to be
$$
\tau_{n}^{\iinfty}(\cW):=\sum \epsilon_p\cdot t_p + \sum \omega(W_J)\cdot J^\iinfty\in\cT^{\iinfty}_{n}
$$
where the first sum is over all order $n$ intersections $p$ and the second sum is over all order $n/2$
Whitney disks $W_J$ with twisting $\omega(W_J)\in\Z$. For $n=0$, recall from \ref{subsec:order-zero-w-towers-and-ints} above our notational convention that $W_j$ denotes $A_j$, and that $\omega(A_j)\in\Z$ is the relative Euler number of the normal bundle of $A_j$ with respect to the given framing of $\partial A_j$ as in \ref{subsec:twisted-w-disks} .

\end{defn}

By splitting the twisted Whitney disks, as explained in subsection~\ref{subsec:split-w-towers} below, 
for $n>0$ we may actually assume that all non-zero $\omega(W_J)\in\{\pm 1\}$, just like the signs $\epsilon_p$. 


As in the framed case, the vanishing of $\tau_{n}^{\iinfty}$ is sufficient for the existence of a 
twisted Whitney tower of order $(n+1)$, and the proof  of Theorem~\ref{thm:twisted} in Section~\ref{sec:proof-twisted-thm} will be based on the following analogue of the framed order-raising Theorem~\ref{thm:framed-order-raising-on-A} to the twisted setting: 
\begin{thm}\label{thm:twisted-order-raising-on-A}
If a collection $A$ of properly immersed surfaces in a simply connected $4$--manifold supports an order $n$ twisted Whitney tower $\cW$ with $\tau_n^\iinfty(\cW)=0\in\cT^\iinfty_n$, then $A$ is regularly homotopic (rel $\partial$) to 
$A'$ which supports an order $n+1$ twisted Whitney tower.
\end{thm}
The proof of Theorem~\ref{thm:twisted-order-raising-on-A} is given in Section~\ref{sec:proof-twisted-thm} below.

Proofs of the ``order-raising'' Theorems \ref{thm:twisted-order-raising-on-A} and \ref{thm:framed-order-raising-on-A} (and its strengthening Theorem~\ref{thm:framed-order-raising-mod-Delta} below) depend on
realizing the relations in the target groups by controlled manipulations of Whitney towers. The next two subsections introduce combinatorial
notions useful for describing the algebraic effect of such geometric constructions. 

\emph{For the rest of this section we assume our Whitney towers are of positive order for convenience of notation.}

\subsection{Intersection forests}\label{subsec:int-forests}
Recall that the trees associated to intersections and Whitney disks in a Whitney tower can be considered to be immersed in the Whitney tower, with vertex orientations induced by the Whitney tower orientation, as in Figure~\ref{WdiskIJandIJKint-fig}. 

\begin{defn}\label{def:intersection forests}
The \emph{intersection forest} $t(\cW)$ of a framed Whitney tower $\cW$ is the disjoint union of signed trees associated to all unpaired intersections $p$ in $\cW$: 
 \[
 t(\cW)=\amalg_p\ \epsilon_p \cdot  t_p
 \]
with $\epsilon_p$ the sign of the intersection point $p$.
For $\cW$ of order~$n$, we can think of the signed order $n$ trees in $t(\cW)$ as an ``abelian word'' in the generators $\pm t_p$ which represents 
$\tau_n(\cW)\in\cT_n$.  
More precisely, $t(\cW)$ is an element of the free abelian monoid, with unit $\emptyset$, generated by (isomorphism classes of) signed trees, trivalent, labeled and vertex-oriented as usual. We emphasize that there are no cancellations or other relations here. 

\begin{rem}
In the older papers \cite{CST,S1,ST2} we referred to $t(\cW)$ as the ``geometric intersection tree'' (and to the group element
$\tau_n(\cW)$ as the order $n$ intersection ``tree'', rather than ``invariant''), but the term ``forest'' better describes
the disjoint union of (signed) trees $t(\cW)$.
\end{rem}

Similarly to the framed case, the \emph{intersection forest} $t(\cW)$ of a {\em twisted} Whitney tower $\cW$ is the disjoint union of signed trees associated to all unpaired intersections $p$ in $\cW$ and integer-coefficient $\iinfty$-trees associated to all non-trivially twisted Whitney disks $W_J$ in $\cW$:
\[
t(\cW)=\amalg_p \ \epsilon_p \cdot  t_p \,\, + \amalg_J \ \omega(W_J)\cdot  J^\iinfty
\]
with $\omega(W_J)\in\Z$ the twisting of $W_J$. Again, there are no cancellations or relations (and the informal ``$+$'' sign in the expression is purely cosmetic).   
\end{defn}
We will see in the next subsection that all the trees can be made to be disjoint in $\cW$, with all non-zero $\omega(W_J)=\pm 1$, so that $t(\cW)$ is also a topological disjoint union which corresponds to an element in the free abelian monoid generated by (isomorphism classes of) signed trees, 
and signed $\iinfty$-trees.

\subsection{Splitting twisted Whitney towers}\label{subsec:split-w-towers}
A framed Whitney tower is \emph{split} if the set of singularities in the interior
of any Whitney disk consists of either a single point, or a single boundary arc of a Whitney disk, or is empty.
This can always be arranged, as observed in Lemma~13 of \cite{ST2} (Lemma~3.5 of \cite{S1}), by performing finger moves along Whitney disks guided by arcs
connecting the Whitney disk boundary arcs. Implicit in this construction is that the finger moves preserve the Whitney disk framings
(by not twisting relative to the Whitney disk that is being split -- see Figure~\ref{twist-split-Wdisk-fig}).
A Whitney disk $W$ is \emph{clean} if the interior of $W$ is embedded and disjoint from the rest of the Whitney tower.
In the setting of twisted Whitney towers, it
will simplify the combinatorics to use ``twisted'' finger moves to similarly split-off twisted Whitney disks 
into $\pm 1$-twisted
clean Whitney disks.

We call a twisted Whitney tower \emph{split} if all of its non-trivially twisted Whitney disks are clean and have twisting 
$\pm 1$, and all of its framed Whitney disks are split in the usual sense (as for framed Whitney towers).
\begin{lem}\label{lem:split-w-tower}
If $A$ supports an order $n$ twisted Whitney tower $\cW$, then $A$ is homotopic (rel $\partial$) to 
$A'$ which supports a split order $n$ twisted
Whitney tower $\cW'$, such that:
\begin{enumerate}
\item The disjoint union of non-$\iinfty$ trees $\amalg_p \ \epsilon_p \cdot  t_p \subset t(\cW)$ is isomorphic to 
the disjoint union of non-$\iinfty$ trees $\amalg_{p'} \ \epsilon_{p'} \cdot  t_{p'} \subset t(\cW')$.
\item Each $\omega(W_J)\cdot  J^\iinfty$ in $t(\cW)$ gives rise to the disjoint union of exactly $|\omega(W_J) |$-many $\pm 1\cdot J^\iinfty$ in $\cW'$, 
where the sign $\pm$ corresponds to the sign of $\omega(W_J)$.
\end{enumerate}
\end{lem}

\begin{proof}
\begin{figure}
\centerline{\includegraphics[width=120mm]{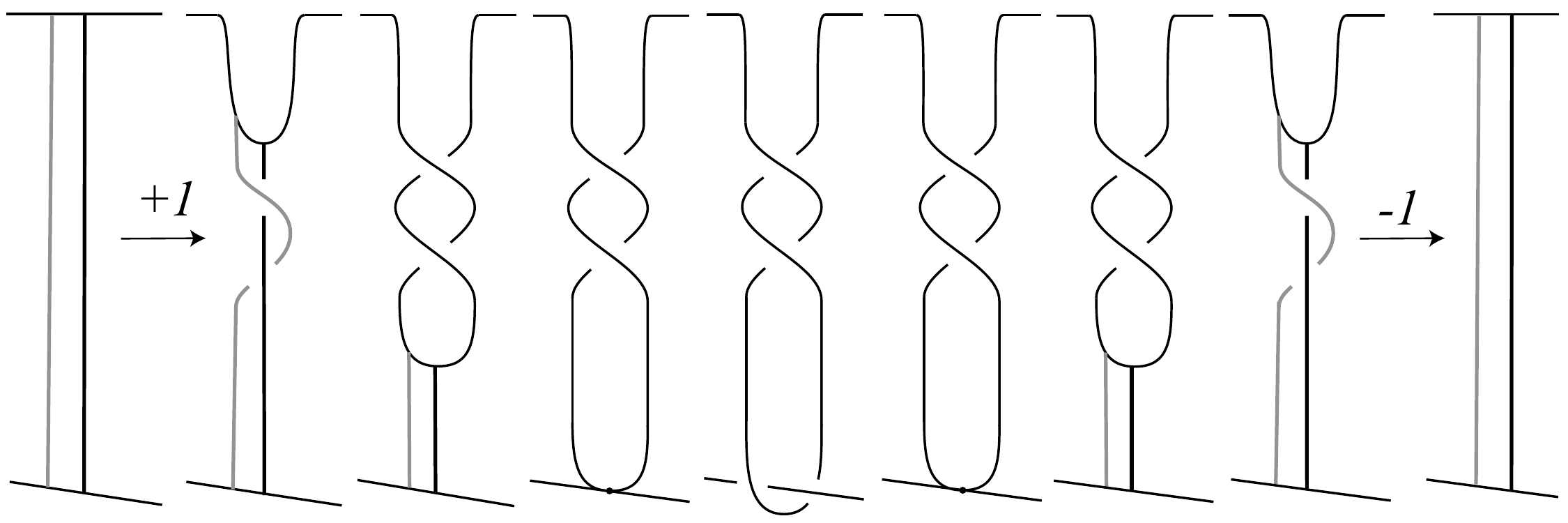}}
         \caption{A neighborhood of a twisted finger move which splits a Whitney disk into two Whitney disks. 
         The vertical black arcs are slices of
         the new Whitney disks, and the grey arcs are slices of extensions of the Whitney sections.
         The finger-move is supported in a neighborhood of an arc in the original Whitney disk running 
         from a point in the Whitney disk boundary on the ``upper'' surface sheet to a point 
         in the Whitney disk boundary on the ``lower'' surface sheet. (Before the finger-move this guiding arc would 
         have been visible in the middle picture as a vertical black arc-slice of the original
         Whitney disk.)}
         \label{twist-split-Wdisk-fig}
\end{figure}
Illustrated in
Figure~\ref{twist-split-Wdisk-fig} is a local picture of a twisted finger move, which splits one Whitney disk into two, while also changing twistings.
If the original Whitney disk in Figure~\ref{twist-split-Wdisk-fig} was framed, then the two new Whitney disks will have twistings $+1$ and $-1$, respectively. In general, if the arc guiding the finger move splits the twisting of the original Whitney disk into $\omega_1$ and $\omega_2$ zeros of the extended Whitney section, then the two new Whitney disks will have twistings $\omega_1+1$ and
$\omega_2-1$, respectively. Thus, by repeatedly splitting off framed corners into $\pm 1$-twisted Whitney disks, any 
$\omega$-twisted Whitney disk ($\omega \in\Z$) can be split into $|\omega |$-many $+1$-twisted or $-1$-twisted clean Whitney disks, together with split framed Whitney disks containing any interior intersections in the original twisted Whitney disk. Combining this with the untwisted splitting of the framed Whitney disks as in Lemma~13 of \cite{ST2} gives the result.
\end{proof}

\begin{rem}\label{rem:locally-flat-and-smooth}
We sketch here a brief explanation of why the smooth and locally flat filtrations are equal. A locally flat surface can be made smooth by a small perturbation, which after introducing cusps as necessary can be assumed to be a regular (locally flat) homotopy. By a general position argument, this regular homotopy can be assumed to be a finite number of finger moves, which are guided by arcs and lead to canceling self-intersection pairs which admit small disjointly embedded Whitney disks (which are `inverses' to the finger moves). These Whitney disks are only locally flat, but can be perturbed to be smooth,
again only at the cost of creating paired self-intersections, and iteration of this process leads to an arbitrarily high-order
smooth sub-Whitney tower pairing all intersections created by the original surface perturbation.
\end{rem}

\section{The realization maps}\label{sec:realization-maps}
This section contains clarifications and proofs of Theorems \ref{thm:R-onto-G} and \ref{thm:twisted} from the introduction
which state the existence of surjections $R_n\colon\cT_n\to\W_n$ and $R^\iinfty_n\colon\cT^\iinfty_n\to\W^\iinfty_n$ for all $n$, in particular exhibiting the 
sets $\W_n$ and $\W^\iinfty_n$ as finitely generated abelian groups under connected sum. 

All proofs in this section apply in the reduced setting as well, and the constructions
described here also define the surjections $\widetilde{R}_n\colon\widetilde{\cT}_n\to\W_n$ described
in the introduction.

Recall that our manifolds are assumed oriented, but orientations are suppressed from the discussion as much as possible.
In the following an orientation is fixed once and for all on $S^3$; and a \emph{framed link} has oriented components, each equipped with a nowhere-vanishing 
normal section.
\begin{defn}\label{def:links-bounding-towers}
A framed link $L\subset S^3=\partial B^4$ \emph{bounds} an order $n$ Whitney tower $\cW$ if
$\cW\subset B^4$ is an order $n$ Whitney tower whose order zero surfaces are immersed disks bounded by the components of $L$, as in Definition~\ref{def:framed-tower}. 

Similarly, a framed link $L\subset S^3=\partial B^4$ \emph{bounds} an order $n$ twisted Whitney tower $\cW$ if
$\cW\subset B^4$ is an order $n$ twisted Whitney tower whose order zero surfaces are immersed disks bounded by the components of $L$, as in Definition~\ref{def:twisted-W-towers}. 
\end{defn}

\begin{defn}\label{def:wtc}
For $n\geq 1$, framed links $L_0$ and $L_1$ in $S^3$ are \emph{Whitney tower concordant of order $n$} if the $i$th components of $L_0\subset S^3\times\{0\}$ and $-L_1\subset S^3\times\{1\}$ cobound an immersed annulus $A_i$ for each $i$ such that the $A_i$ are transverse and support an order $n$ Whitney tower.
If the $A_i$ support a \emph{twisted} order $n$ Whitney tower then 
$L_0$ and $L_1$ are said to be \emph{twisted Whitney tower concordant of order $n$}.
\end{defn}
Note that a (twisted) Whitney tower concordance preserves framings on on $L_0$ and $L_1$ (as links in $S^3$)
since all self-intersections of the $A_i$ are paired by Whitney disks in any (twisted) Whitney tower of order $n\geq 1$.

Recall from the introduction that the set of $m$-component framed links in $S^3$ which bound order $n$ (twisted) Whitney towers in $B^4$ is denoted by $\bW_n=\bW_n(m)$ (resp. $\bW^\iinfty_n$); and and the quotient of $\bW_n$ by the equivalence relation of order $n+1$ (twisted) Whitney tower concordance is 
denoted by $\W_n$ (resp. $\W^\iinfty_n$).

Throughout this section the twisted setting mirrors the framed setting, with discussions and arguments given simultaneously.

We first need to show our essential criterion, Corollary~\ref{cor:tau=w-concordance}, for links to represent equal elements in the associated graded $\W_n$:
Links $L_0$ and $L_1$ in $\bW_n$ represent the same element of $\W_n$
if and only if there exist order $n$ Whitney towers $\cW_i$ in $B^4$ with $\partial\cW_i=L_i$ and $\tau_n(\cW_0)=\tau_n(\cW_1)\in\cT_n$.

\begin{proof}[Proof of Corollary~\ref{cor:tau=w-concordance}] 
If $L_0$ and $L_1$ are equal in $\W_n$ then they cobound $A$ supporting an order $n+1$ Whitney tower $\cV$ in $S^3\times I$, and any order $n$ Whitney tower $\cW_1$ in $B^4$ bounded by $L_1$ can be extended by $\cV$ to 
form an order $n$ Whitney tower $\cW_0$ in $B^4$ bounded by $L_0$, with 
$\tau_n(\cW_0)=\tau_n(\cW_1)\in\cT_n$
since $\tau_n(\cV)$ vanishes. 

Conversely, suppose that $L_0$ and $L_1$ bound order $n$ Whitney towers $\cW_0$ and $\cW_1$ in $4$--balls $B_0^4$ and $B_1^4$, with 
$\tau_n(\cW_0)=\tau_n(\cW_1)$. Then constructing $S^3\times I$
as the connected sum $B_0^4\# B_1^4$ (along balls in the complements of $\cW_0$ and $\cW_1$), and tubing together the corresponding order zero disks of $\cW_0$ and $\cW_1$, and taking the union of the Whitney disks in $\cW_0$ and 
$\cW_1$, yields a collection $A$ of properly immersed annuli connecting $L_0$ and $L_1$ and supporting an order $n$ Whitney tower $\cV$.  Since the orientation of the ambient $4$--manifold has been reversed for one of the original Whitney towers, say $\cW_1$, which results in a global sign change for 
$\tau_n(\cW_1)$, it follows that $\cV$ has vanishing order $n$ intersection invariant:
$$
\tau_n(\cV)=\tau_n(\cW_0)-\tau_n(\cW_1)=\tau_n(\cW_0)-\tau_n(\cW_0)=0\in\cT_n
$$
So by Theorem~\ref{thm:framed-order-raising-on-A}, $A$ is homotopic (rel $\partial$) to $A'$ supporting
an order $n+1$ Whitney tower, and hence $L_0$ and $L_1$ are equal in $\W_n$.
\end{proof} 

\begin{rem}\label{rem:tau=w-concordance}
The analogous statement and proof of Corollary~\ref{cor:tau=w-concordance} holds in the twisted case (with Theorem~\ref{thm:twisted-order-raising-on-A} playing the role of Theorem~\ref{thm:framed-order-raising-on-A}). For this case, we'll spell out the statement carefully but in several instances below we will just state that the twisted case is analogous:
Links $L_0$ and $L_1$ in $\bW_n^\iinfty$ represent the same element of $\W^\iinfty_n$
if and only if there exist order $n$ twisted Whitney towers $\cW_0$ and $\cW_1$ in $B^4$ bounded by $L_0$ and $L_1$ respectively such that
$\tau^\iinfty_n(\cW_0)=\tau^\iinfty_n(\cW_1)\in\cT^\iinfty_n$.
\end{rem}

\begin{rem}\label{rem:reduced-tau=w-concordance}
Remark~\ref{rem:tau=w-concordance} similarly applies to the reduced setting by Theorem~\ref{thm:framed-order-raising-mod-Delta} below, although we will omit further reference to $\widetilde{\cT}$ in this section.
\end{rem}

\subsection{Band sums of links}\label{subsec:band-sum}
The \emph{band sum} $L\#_\beta L'\subset S^3$
of oriented $m$-component links $L$ and $L'$ along bands $\beta$ is defined as follows: Form $S^3$ as the connected sum of $3$--spheres containing $L$ and $L'$ along balls in the link complements.  Let
$\beta$ be a collection of disjointly embedded oriented bands joining like-indexed link components such that the band orientations are compatible with the link orientations. Take the usual connected sum of each pair of components along the corresponding band. Although it is well-known that the concordance class of $L\#_\beta L'$ depends in general on $\beta$, it turns out that the image of 
$L\#_\beta L'$ in $\W_n$ (or in $\W^\iinfty_n$) does not depend on $\beta$:

\begin{lem}\label{lem:link-sum-well-defined}   
For links $L$ and $L'$ representing elements of $\W_n$, any band sum $L\#_\beta L'$ represents an element of 
$\W_n$ which only depends on the equivalence classes of $L$ and $L'$ in $\W_n$.	The same statement holds in $\W^\iinfty_n$.
\end{lem}

\begin{proof} We shall only give the proof in the framed case, the twisted case is analogous.
If $L_0$ and $L_1$ represent the same element of $\W_n$, 
and if $L'_0$ and $L'_1$ represent the same element of $\W_n$, then
by Corollary~\ref{cor:tau=w-concordance} above, for $i=0,1$, there are order $n$ Whitney towers $\cW_i$
and $\cW'_i$ bounding $L_i$ and $L'_i$ such that $\tau_n(\cW_0)=\tau_n(\cW_1)$ and $\tau_n(\cW'_0)=\tau_n(\cW'_1)$.
By Lemma~\ref{lem:exists-tower-sum} just below, $L_i\#_{\beta_i} L'_i$ bounds $\cW_i^\#$ for $i=0,1$, with
$$
\tau_n(\cW_0^\#)=\tau_n(\cW_0)+\tau_n(\cW'_0)=\tau_n(\cW_1)+\tau_n(\cW'_1)=\tau_n(\cW_1^\#)
$$
so again by Corollary~\ref{cor:tau=w-concordance}, $L_0\#_{\beta_0} L'_0$ is order $n+1$ Whitney tower concordant to 
$L_1\#_{\beta_1} L'_1$, hence $L_0\#_{\beta_0} L'_0$ 
and $L_1\#_{\beta_1} L'_1$ represent the same element of $\W_n$.
\end{proof}

\begin{lem}\label{lem:exists-tower-sum}
If $L$ and $L'$ bound order $n$ (twisted) Whitney towers $\cW$ and $\cW'$ in $B^4$, then for any $\beta$ there exists an order $n$ (twisted)
Whitney tower $\cW^\#\subset B^4$ bounded by $L\#_\beta L'$, such that 
$t(\cW^\#)=t(\cW)\amalg t(\cW')$, where $t(\cV)$ denotes the intersection forest of a Whitney tower $\cV$ as above in subsection~\ref{subsec:int-forests}.
\end{lem}
\begin{proof}
Let $B$ and $B'$ be the $3$--balls in the link complements used to form the $S^3$ containing $L\#_\beta L'$.  Then gluing together the two $4$--balls containing $\cW$ and $\cW'$ along $B$ and $B'$ forms $B^4$ containing $L\#_\beta L'$ in its boundary. Take $\cW^\#$ to be the boundary band sum of $\cW$ and $\cW'$ along the order zero disks guided by the bands 
$\beta$, with the interiors of the bands perturbed slightly into the interior of $B^4$. 
It is clear that $t(\cW^\#)$ is just the disjoint union $t(\cW)\amalg t(\cW')$ since no new singularities have been created.
\end{proof}

\subsection{The realization maps}\label{subsec:realization-maps}
The realization maps $R_n$ are defined as follows: Given any group element $g\in\cT_n$, by Lemma~\ref{lem:realization-of-geometric-trees} just below there exists an $m$-component link $L\subset S^3$ bounding
an order $n$ Whitney tower $\cW\subset B^4$ such that $\tau_n(\cW)=g\in\cT_n$.
Define $R_n(g)$ to be the class determined by $L$ in $\W_n$.
This is well-defined (does not depend on the choice of such $L$) by Corollary~\ref{cor:tau=w-concordance}. The twisted realization map $R^\iinfty_n$ is defined the same way using twisted Whitney towers.

\begin{lem}\label{lem:realization-of-geometric-trees}
For any disjoint union $\amalg_p\ \epsilon_p \cdot  t_p \,\, + \amalg_J\ \omega (W_J) \cdot  J^\iinfty$ there exists an $m$-component link $L$ bounding a twisted  Whitney tower $\cW$
with intersection forest $t(\cW)= \amalg_p\ \epsilon_p \cdot  t_p \,\, + \amalg_J \ \omega (W_J) \cdot  J^\iinfty$. 
If the disjoint union contains no $\iinfty$-trees then all Whitney disks in $\cW$ are framed.
\end{lem}
Note that if in the disjoint union all non-$\iinfty$ trees are order at least $n$ and all $\iinfty$-trees are order at least $n/2$ then  $\cW$ will have order $n$.

\begin{proof}
It suffices to consider the cases where the disjoint union consists of just a single (signed) tree or $\iinfty$-tree  since by Lemma~\ref{lem:exists-tower-sum} any sum of such trees can then be realized by band sums of links.

The following algorithm, in the untwisted case, is the algorithm called "Bing-doubling along a tree" by Cochran and used in Section 7 of \cite{C} and Theorem 3.3 of \cite{C1} to produce links in $S^3$  with prescribed (first non-vanishing) Milnor invariants.

\begin{figure}
\centerline{\includegraphics[scale=.325]{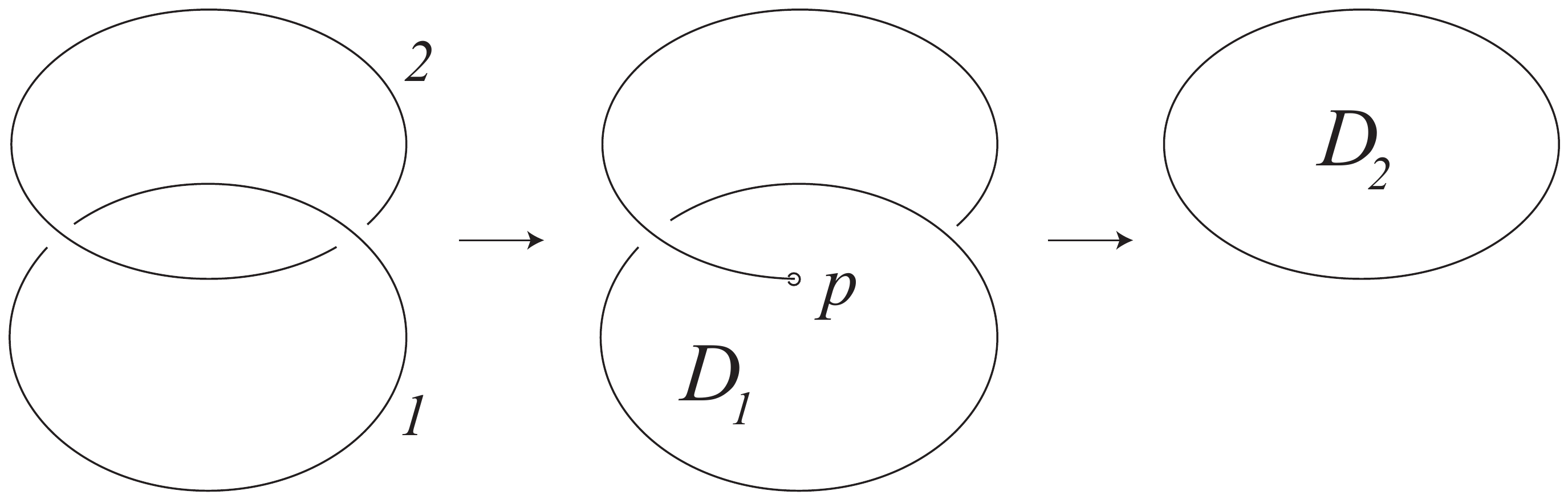}}
         \caption{Pushing into $B^4$ from left to right: A Hopf link in $S^3=\partial B^4$ bounds embedded disks $D_1\cup D_2\subset B^4$ which intersect in a point $p$, with $t_p=\langle 1,2 \rangle$. }
         \label{fig:Hopf-disk}
\end{figure}

\textbf{Realizing order zero trees and $\iinfty$-trees.}
A 0-framed Hopf link bounds an order zero Whitney tower $\cW=D_1\cup D_2\subset B^4$, where the two embedded disks $D_1$ and $D_2$ have a single interior intersection point $p$ with $t_p=\langle 1,2 \rangle = 1 -\!\!\!-\!\!\!- \,2 $ (see Figure~\ref{fig:Hopf-disk}).
Assuming appropriate fixed orientations of $B^4$ and $S^3$, the sign $\epsilon_p$ associated to $p$ is the usual sign of the Hopf link. So taking a 0-framed $(m-2)$-component trivial link together with a Hopf link (as the $i$th and $j$th components) gives an $m$-component link $L$ bounding
$\cW$ with $t(\cW)=\epsilon_p\cdot\langle i,j \rangle = \epsilon_p\cdot i -\!\!\!- \,j $, for any 
$\epsilon_p=\pm 1$, and $i\neq j$.

To realize the tree $\pm\ i -\!\!\!-\!\!\!- \,i $, we can use the unlink with framings 0, except that the component labeled by the index $i$ has framing $\pm 2$. Similarly, if the component has framing $\pm 1$ then the resulting tree is $\pm \ \iinfty -\!\!\!-\!\!\!-\,i $.

\begin{figure}
\centerline{\includegraphics[scale=.4]{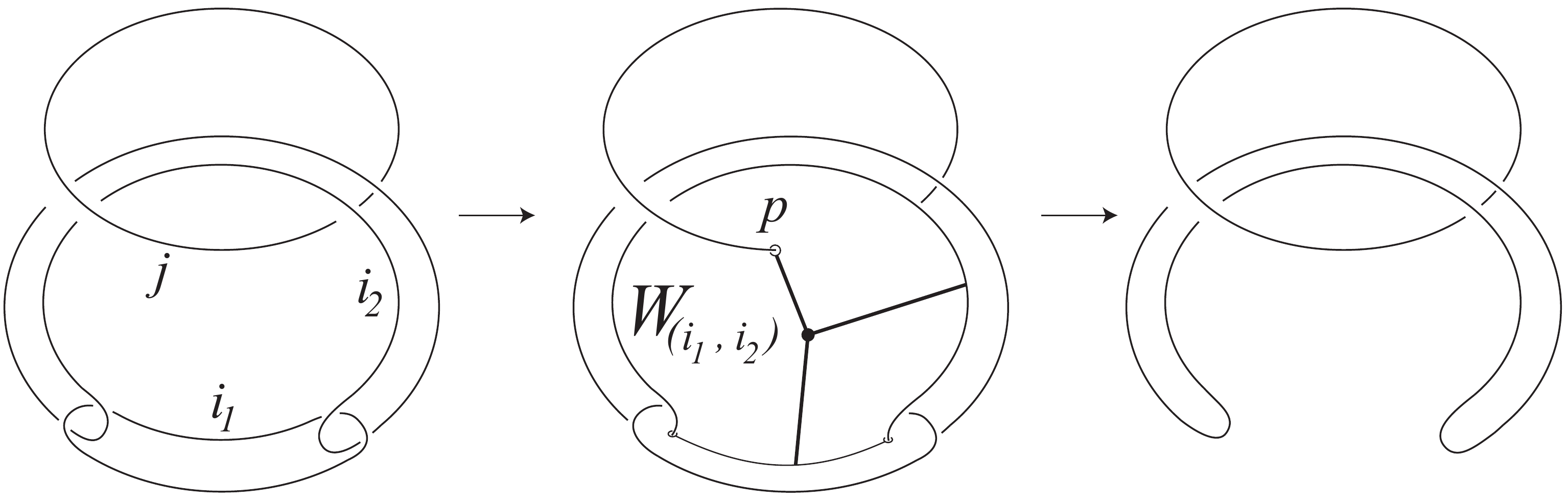}}
         \caption{Pushing into $B^4$ from left to right: The disks $ D_{i_2}$ and $D_j$ extend to the
         right-most picture where they are completed by capping off the unlink. The disk $D_{i_1}$ only extends to the middle picture where the intersections between $D_{i_1}$ and $ D_{i_2}$ are paired by the Whitney disk $W_{(i_1,i_2)}$, that has a single interior intersection $p\in W_{(i_1,i_2)}\cap D_j$ with $t_p=\langle (i_1,i_2),j \rangle$.}
         \label{fig:Borromean-Bing-Hopf}
\end{figure}

\textbf{Realizing order $1$ trees.}
Consider now a link $L$ whose $i$th and $j$th components form a Hopf link $L^i\cup L^j$ bounding disks
$D_i\cup D_j\subset B^4$ with transverse intersection $p=D_i\cap D_j$. Assume that $D_i\cup D_j$ extends
to an order zero Whitney tower $\cW$ bounded by $L$ with $t(\cW)=\epsilon_p\cdot t_p=\epsilon_p\cdot\langle i,j \rangle$.

Replacing $L^i$ by an untwisted Bing-double
$L^{i_1}\cup L^{i_2}$
results in a new sublink of Borromean rings $L^{i_1}\cup L^{i_2}\cup L^j$ bounding disks $D_{i_1}\cup D_{i_2}\cup D_j$
as indicated in Figure~\ref{fig:Borromean-Bing-Hopf}, with
$D_{i_1}$ and $ D_{i_2}$ intersecting in a canceling pair of intersections paired by an order $1$ Whitney disk $W_{(i_1,i_2)}$, which can be formed from $D_i$ with a small collar removed,
so that $W_{(i_1,i_2)}$ has a single intersection with $D_j$ corresponding to the original $p=D_i\cap D_j$.
(One can think of $D_{i_1}$ and $ D_{i_2}$ as being formed by the trace of the obvious 
pulling-apart homotopy that shrinks $L^{i_1}$ and $L^{i_2}$ down in a tubular neighborhood of $L^i$, 
with the canceling pair of intersections between $D_{i_1}$ and $D_{i_2}$ being created as the clasps are pulled apart.)

The effect of this Bing-doubling operation on the intersection forest is that the original order zero
$t_p=\langle i,j \rangle$ has given rise to the order~$1$ tree $\langle (i_1,i_2),j \rangle$. 
Switching the orientation on one of the new components changes the sign of $p$, as can be checked using our orientation conventions.
By relabeling and/or banding together components of this new link any labels on this order~$1$ tree
can be realized. Since the doubling was untwisted, $W_{(i_1,i_2)}$ is framed (see Figures \ref{fig:Bing-unlink-W-disk} and \ref{fig:Bing-unlink-W-disk-twisting}), so the Whitney tower
bounded by the new link
is order $1$. 

\begin{figure}
\centerline{\includegraphics[scale=.35]{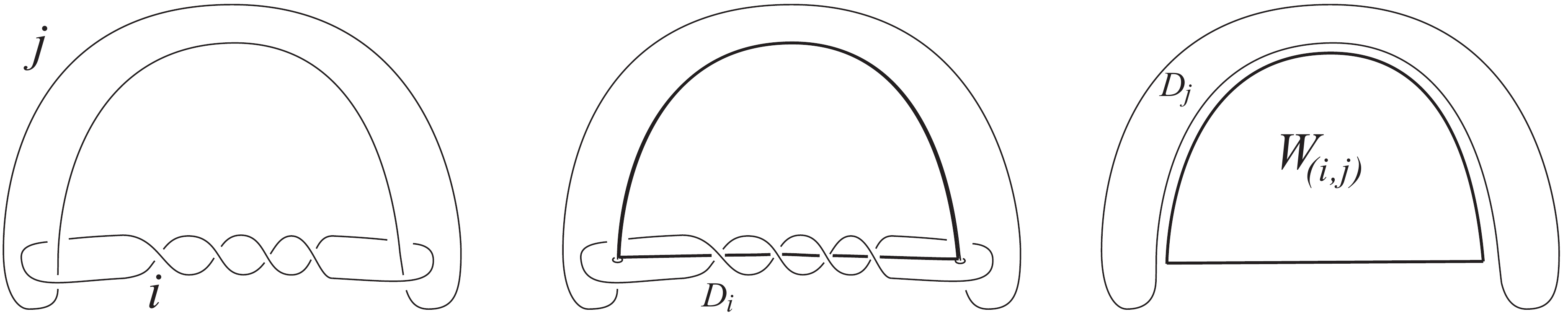}}
         \caption{Pushing into $B^4$ from left to right: An $i$- and $j$-labeled $n$-twisted Bing-double (case $n=2$) of the unknot in 
         $S^3=\partial B^4$
         bounds disks $D_i$ and $D_j$ whose intersections are paired by a Whitney disk $W_{(i,j)}$. 
         $D_j$ extends to the right-hand picture but $D_i$ only extends to the middle picture, where the boundary of $W_{(i,j)}$ is indicated by the dark arcs. The rest of $W_{(i,j)}$ extends into the right-hand picture where disjointly embedded disks bounded by the unlink complete both $W_{(i,j)}$ and $D_j$. The interior of $W_{(i,j)}$ is embedded and disjoint from both $D_i$ and $D_j$. Figure~\ref{fig:Bing-unlink-W-disk-twisting} shows that $W_{(i,j)}$ is twisted, with $\omega(W_{(i,j)})=n$.}
         \label{fig:Bing-unlink-W-disk}
\end{figure}

\textbf{Realizing order $n$ trees.}
Since any order $n$ tree can be gotten from some order $n-1$ tree by attaching two new edges to a univalent vertex as in the previous paragraph, it follows inductively that 
any order $n$ tree is the intersection forest of a Whitney tower bounded by some link. (First create a distinctly-labeled tree of the desired `shape' by doubling, then correct the labels by interior band-summing.)

\textbf{Realizing $\iinfty$-trees of order $1$.}
As illustrated (for the case $n=2$) in Figures~\ref{fig:Bing-unlink-W-disk} and \ref{fig:Bing-unlink-W-disk-twisting}, the $n$-twisted Bing-double
of the unknot (with components labeled $i$ and $j$) bounds an order $2$ twisted Whitney tower $\cW$ with 
$t(\cW)=n\cdot ( i,j )^\iinfty=n\cdot  \iinfty \!-\!\!\!\!\!-\!\!\!\!<^{\,i}_{\,j}$. Banding together the two components
would yield a knot realizing $(i,i)^\iinfty$. 

\begin{figure}
\centerline{\includegraphics[scale=.4]{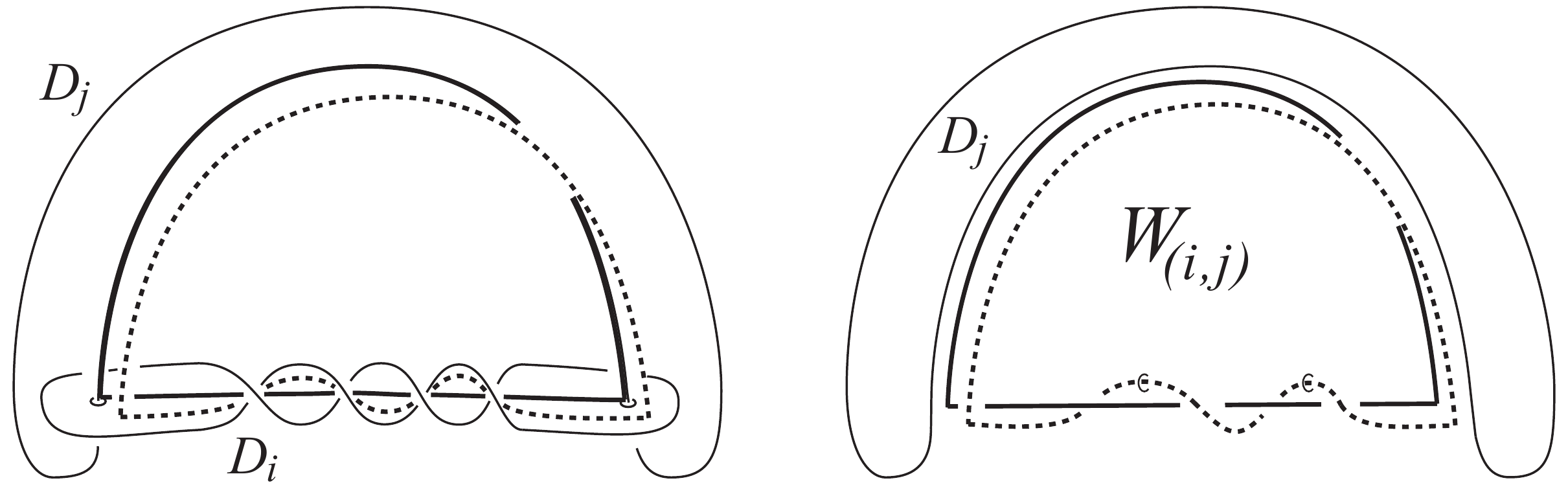}}
         \caption{The Whitney section over $\partial W_{(i,j)}$ (from Figure~\ref{fig:Bing-unlink-W-disk}) is indicated by the dashed arcs on the left. The twisting 
         $\omega (W_{(i,j)})=n$ (the obstruction to extending the Whitney section across the Whitney disk) corresponds to the 
         $n$-twisting of the Bing-doubling operation.}
         \label{fig:Bing-unlink-W-disk-twisting}
\end{figure}

\textbf{Realizing $\iinfty$-trees of order $n$.}
By applying iterated untwisted Bing-doubling operations to the $i$- and $j$-labeled components 
of the order $1$ case, one can construct for any rooted tree $( I,J )$ a link bounding a twisted Whitney tower
$\cW$ with $t(\cW)=n \cdot ( I,J )^\iinfty$. For instance, if in the construction of Figure~\ref{fig:Bing-unlink-W-disk} the $j$-labeled link component is replaced by an untwisted Bing-double, then the disk $D_j$ in that construction would be replaced by a (framed) Whitney disk $W_{(j_1,j_2)}$, and the $n$-twisted $W_{(i,j)}$ would be replaced by an $n$-twisted
$W_{(i,(j_1,j_2))}$.   (As for non-$\iinfty$ trees above, first create a distinctly-labeled tree of the desired `shape' by doubling, then correct the labels by interior band-summing.) 
\end{proof}


\subsection{Proofs of Theorem~\ref{thm:R-onto-G} and Theorem~\ref{thm:twisted}}\label{subsec:realization-maps}
Recall the content of Theorem~\ref{thm:R-onto-G}: The realization maps $R_n: \cT_n \to\W_n$ are epimorphisms, 
with the group operation on $\W_n$ given by band sum $\#$. Moreover, $\W_n$ is the set of framed links $L\in\bW_n$ modulo the relation that $[L_1]=[L_2]\in\W_n$ if and only if $L_1 \# -L_2$ lies in $\bW_{n+1}$, for some choice of connected sum $\#$, where $-L$ is the mirror image of $L$ with reversed framing.
The content of Theorem~\ref{thm:twisted} in the twisted setting is analogous.

From Lemma~\ref{lem:link-sum-well-defined} the band sum of links gives well-defined operations in $\W_n$ and $\W^\iinfty_n$ which are clearly associative and commutative, with the $m$-component unlink representing an identity element.  The realization maps are homomorphisms by Lemma~\ref{lem:exists-tower-sum} and surjectivity is proven as follows: Given any link $L\in \bW_n$, choose a Whitney tower $\cW$ of order $n$ with boundary $L$ and compute $\tau:=\tau_n(\cW)$. Then take $L':= R_n(\tau)$, a link that's obviously in the image of $R_n$ and for which we know a Whitney tower $\cW'$ with boundary $L'$ and $\tau(\cW') = \tau$.
By Corollary~\ref{cor:tau=w-concordance} it follows that $L$ and $L'$ represent the same element in $\W_n$.

Considering the second ``Moreover...'' statements of Theorem~\ref{thm:R-onto-G} and Theorem~\ref{thm:twisted}, first assume that
$L_0$ and $L_1$ represent the same element of $\W_n$ (resp. $\W^\iinfty_n$). Then by Corollary~\ref{cor:tau=w-concordance}, there exist order $n$ (twisted) Whitney towers $\cW_0$ and $\cW_1$ in $B^4$ bounded by $L_0$ and $L_1$ respectively such that
$\tau_n(\cW_0)=\tau_n(\cW_1)\in\cT_n$ (resp. $\tau^\iinfty_n(\cW_0)=\tau^\iinfty_n(\cW_1)\in\cT^\iinfty_n$). We want to show that $L_0\#-L_1$ bounds an order~$n+1$ (twisted) Whitney tower, which will follow from Lemma~\ref{lem:exists-tower-sum} and the ``order-raising'' Theorem~\ref{thm:framed-order-raising-on-A} (respectively Theorem~\ref{thm:twisted-order-raising-on-A}) if $-L_1$ bounds an order $n$ (twisted) Whitney tower $\overline{\cW_1}$ such that $\tau_n(\overline{\cW_1})=-\tau_n(\cW_1)\in\cT_n$ 
(resp. $\tau^\iinfty_n(\overline{\cW_1})=-\tau^\iinfty_n(\cW_1)\in\cT^\iinfty_n$). If $r$ denotes the reflection on $S^3$ which sends $L_1$ to 
$-L_1$, then the product $r\times\id$ of $r$ with the identity is an involution on $S^3\times I$, and the image $r\times\id(\cW_1)$
of $\cW_1$ is such a $\overline{\cW_1}$. To see this, observe that $r\times\id$ switches the signs of all transverse intersection points,
and is an isomorphism on the oriented trees in $\cW_1$; and hence switches the signs of all Whitney disk framing obstructions (which can be computed as intersection numbers between Whitney disks and their push-offs) -- note that $r\times\id$ is only being applied to 
$\cW_1$, while $S^3\times I$ is fixed.   

For the other direction of the ``Moreover...'' statements, assume that $L_0\#-L_1\subset S^3$ bounds an order~$n+1$ (twisted) Whitney tower $\cW\subset B^4$. By the definition of connected sum, $S^3$ decomposes as the union of two disjoint $3$--balls $B_0$ and $B_1$ containing $L_0$ and $-L_1$, joined together by the $S^2\times I$ through which passes the bands guiding the connected sum. 
Taking another $4$--ball with the same decomposition of its boundary $3$--sphere, and gluing the $4$--balls together by identifying the 
boundary $2$--spheres of the $3$--balls, and identifying the $S^2\times I$ subsets by the identity map, forms $S^3\times I$ containing
an order $n+1$ (twisted) Whitney tower concordance between $L_0$ and $-L_1$ which consists of $\cW$ together with the parts of the connected-sum bands that are contained in $S^2\times I$.


\section{Proof of Theorem~\ref{thm:twisted-order-raising-on-A} and the twisted IHX lemma}\label{sec:proof-twisted-thm}
This section contains a proof of the ``twisted order-raising'' Theorem~\ref{thm:twisted-order-raising-on-A} of Section~\ref{sec:w-towers}, which was used (along with Corollary~\ref{cor:tau=w-concordance}) in Section~\ref{sec:realization-maps} to prove 
Theorem~\ref{thm:twisted} of the introduction. A key step in the proof involves a geometric realization of the 
twisted IHX relation as described in Lemma~\ref{lem:twistedIHX} below. 

At the end of the section, the proof of Proposition~\ref{prop:exact sequence} is completed by Lemma~\ref{lem:boundary-twisted-IHX} in \ref{subsec:boundary-twisted-IHX-lemma} which shows how any order $2n$ twisted Whitney tower can be converted into an order $2n-1$ framed Whitney tower. 

\subsection{Proof of Theorem~\ref{thm:twisted-order-raising-on-A}}\label{subsec:proof-of-twisted-order-raising-thm}
Recall the statement of Theorem~\ref{thm:twisted-order-raising-on-A}: 
If a collection $A$ of properly immersed surfaces in a simply connected $4$--manifold supports an order $n$ twisted Whitney tower $\cW$
with $\tau^\iinfty_n(\cW)=0\in\cT^\iinfty_n$, then $A$ is regularly homotopic (rel $\partial$) to $A'$ supporting an order $n+1$ twisted Whitney tower.

Recall also from subsection~\ref{subsec:int-forests} that the intersection forest $t(\cW)$ of an order 
$n$ twisted Whitney tower $\cW$
is a disjoint union of signed trees which can be considered to be immersed in $\cW$. The order $n$ trees in $t(\cW)$ (together with the order $n/2$ 
$\iinfty$-trees if $n$ is even) represent $\tau_n^\iinfty(\cW)\in\cT^\iinfty_n$, and the proof of Theorem~\ref{thm:twisted-order-raising-on-A} involves controlled manipulations of $\cW$ which first convert $t(\cW)$ into ``algebraically canceling'' pairs of isomorphic trees with opposite signs, and then exchange these for ``geometrically canceling'' intersection points which are paired by a new layer of Whitney disks. We pause here to clarify these
notions:


\subsubsection{Algebraic versus geometric cancellation}\label{subsubsec:alg-vs-geo-cancellation}
If Whitney disks $W_I$ and $W_J$ in $\cW$ intersect transversely in a pair of points $p$ and $p'$, then $t_p$ and $t_{p'}$ are isomorphic (as labeled, oriented trees). If $p$ and $p'$ have opposite signs, and if the ambient $4$--manifold is simply connected, then there exists a Whitney disk $W_{(I,J)}$ pairing $p$ and $p'$,
and we say that $\{\,p\,,\,p'\,\}$ is a \emph{geometrically canceling} pair.
In this setting we also refer to $\{\,\epsilon_p\cdot t_p\,,\,\epsilon_{p'}\cdot t_{p'}\,\}$ as a geometrically canceling pair of signed trees in $t(\cW)$ (regarding them as subsets of $\cW$ associated to the geometrically canceling pair of points).

On the other hand, given transverse intersections $p$ and $p'$ in $\cW$ with $ t_p = t_{p'}$ 
(as labeled oriented trees) and $\epsilon_p=-\epsilon_{p'}$, we say that $\{\,p\,,\,p'\,\}$ is an \emph{algebraically canceling} pair of intersections, and similarly call $\{\,\epsilon_p\cdot t_p\,,\,\epsilon_{p'}\cdot t_{p'}\,\}$ an algebraically canceling pair of signed trees in $t(\cW)$. Changing the orientations at a \emph{pair} of trivalent vertices in any tree $t_p$ does not change its value in $\cT$ by the AS relations, and (as discussed in 3.4 of \cite{ST2}) such orientation changes can be realized by changing orientations of Whitney disks in $\cW$ together with our orientation conventions (\ref{subsec:w-tower-orientations}).

Any geometrically canceling pair is also an algebraically canceling pair, but the converse is clearly not true as an algebraically canceling pair can have \emph{corresponding trivalent vertices} lying in \emph{different Whitney disks}.  A process for converting algebraically canceling pairs into geometrically canceling pairs by manipulations of the Whitney tower is described in 4.5 and 4.8 of \cite{ST2}.

Similarly, if a pair of twisted Whitney disks $W_{J_1}$ and $W_{J_2}$ have isomorphic (unoriented) trees $J^\iinfty_1$ and $J^\iinfty_2$ with opposite twistings $\omega(W_{J_1})=-\omega (W_{J_2})$, then the Whitney disks form an 
\emph{algebraically canceling} pair (as do the corresponding signed $\iinfty$-trees in $t(\cW)$). Note that the orientations of the $\iinfty$-trees are not relevant here by the independence of $\omega(W)$ from the orientation of $W$ and the symmetry relations in $\cT^\iinfty$. A geometric construction for eliminating algebraically canceling pairs of twisted Whitney disks from a twisted Whitney tower will be described below.


\textbf{Outline of the proof:}
To motivate the proof of Theorem~\ref{thm:twisted-order-raising-on-A} we summarize here how the methods of 
\cite{CST,S1,ST2} (as described in Section~4 of \cite{ST2})
apply in the framed setting to prove the analogous order-raising theorem in framed setting (Theorem~\ref{thm:framed-order-raising-on-A} of Section~\ref{sec:w-towers}): The first part of the proof changes the intersection forest $t(\cW)$ so that all trees occur in algebraically canceling pairs by using the $4$--dimensional IHX construction of \cite{CST} to realize IHX relators, and by adjusting Whitney disk orientations as necessary to realize AS relations. The second part of the proof uses the Whitney move IHX construction of \cite{S1}
to ``simplify'' the shape of the algebraically canceling pairs of trees. Then the third part of the proof uses controlled
homotopies to exchange the simple algebraic canceling pairs for geometrically canceling intersection points which are paired by a new layer of Whitney disks as described in 4.5 of \cite{ST2}. 

Extending these methods to the present twisted setting will require two variations: realizing the new relators in 
$\cT_n^\iinfty$, and achieving an analogous geometric cancellation for twisted Whitney disks corresponding to algebraically canceling pairs of (simple) $\iinfty$-trees. We will concentrate on these new variations, referring the reader to   
\cite{CST,S1,ST2} for the other parts just mentioned.


\textbf{Notation and conventions:}

By Lemma~\ref{lem:split-w-tower} it may be assumed that $\cW$ is split at each stage of the constructions throughout the proof, so that all trees in $t(\cW)$ are embedded in $\cW$.  

In spite of modifications, $\cW$ will not be renamed during the proof.

Throughout this section we
will notate elements of $t(\cW)$ as formal sums, representing disjoint union by juxtaposition.

Note that if $\cW$ is an order $n$ twisted Whitney tower, then the intersection forest $t(\cW)$ may contain
higher order trees and $\iinfty$-trees in addition to those representing $\tau_n^\iinfty(\cW)$. These higher-order
elements of $t(\cW)$ can be ignored throughout the proof for the following reasons: On the one hand, in a split $\cW$ all the 
constructions leading to the elimination of unpaired order $n$ intersections (and twisted Whitney disks of order $n/2$)
of $\cW$
can be carried out away from any higher-order
elements of $t(\cW)$. Alternatively, one could first exchange all twisted Whitney disks of order greater than $n/2$ for unpaired intersections of order greater than $n$ by boundary-twisting (Figure~\ref{boundary-twist-and-section-fig}). Then, all intersections of order greater than $n$ can be converted into into many algebraically canceling pairs of order $n$ intersections by repeatedly ``pushing down'' unpaired intersections
until they reach the order zero disks, as illustrated for instance in Figure~12 of \cite{S2} (assuming, as we may, that $\cW$ contains no Whitney disks of order greater than $n$). 

Thus, we can and will assume throughout the proof that $t(\cW)$ represents $\tau_n^\iinfty(\cW)$.

\textbf{The odd order case:} Given $\cW$ of order $2n-1$ with $\tau^\iinfty_{2n-1}(\cW)=0\in\cT^\iinfty_{2n-1}$, it will suffice to modify
$\cW$ --- while only creating unpaired intersections of order at least $2n$ and twisted Whitney disks of order at 
least $n$ --- so that all order $2n-1$ trees in $t(\cW)$ come in algebraically canceling
pairs of trees (since by
\cite{ST2} the corresponding algebraically canceling pairs of order $2n-1$ intersection points can be
exchanged for geometrically canceling intersections which are
paired by Whitney disks, as mentioned just above).

\begin{figure}
\centerline{\includegraphics[width=125mm]{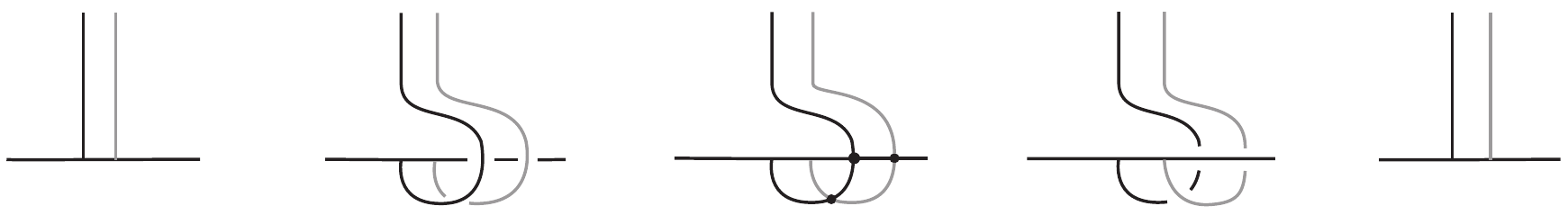}}
         \caption{Boundary-twisting a Whitney disk $W$ changes $\omega (W)$ by $\pm 1$ and creates an intersection point with one of the sheets paired by $W$. The horizontal arcs trace out part of the sheet, the dark non-horizontal arcs
         trace out the newly twisted part of a collar of $W$, and the grey arcs indicate part of the Whitney section over $W$. The bottom-most
         intersection in the middle picture corresponds to the $\pm 1$-twisting created  by the move.}
         \label{boundary-twist-and-section-fig}
\end{figure}

Since $\tau_{2n-1}^\iinfty(\cW)=0\in\cT^\iinfty_{2n-1}$, the intersection forest $t(\cW)$ is in the span of 
IHX and boundary-twist relators, after choosing Whitney disk orientations to realize AS relations as necessary. By locally creating intersection
trees of the form $+I -H  +X$ using  the 4-dimensional geometric IHX theorem of \cite{CST} 
(and by
choosing Whitney disk orientations to realize AS relations as needed), 
$\cW$
can be modified so that all order $2n-1$ trees in $t(\cW)$ either come in algebraically
canceling pairs, or are boundary-relator trees of the form
$\pm\langle (i,J),J \rangle$. 

For each tree of the form 
$t_p=\pm\langle (i,J),J \rangle$ we can create an algebraically canceling
 $t_{p'}=\mp\langle (i,J),J \rangle$ at the cost of only creating order~$n$
$\iinfty$-trees as follows. First use Lemma~14 of \cite{ST2} (Lemma~{3.6} of \cite{S1}) to move
the unpaired intersection point $p$ so that $p\in W_{(i,J)}\cap
W_J$. Now, by boundary-twisting $W_{(i,J)}$ into its supporting
Whitney disk $W'_J$ (Figure~\ref{boundary-twist-and-section-fig}), 
an algebraically canceling intersection $p'\in
W_{(i,J)}\cap W'_J$ can be created at the cost of changing the
twisting $\omega (W_{(i,J)})$ by $\pm 1$. 
Since $\langle (i,J),J \rangle$ has an order $2$ symmetry, the canceling sign can always be realized by a Whitney disk orientation choice. 
This algebraic cancellation of $t_p$ has been achieved at the cost of only adding to $t(\cW)$ the 
order $n$ $\iinfty$-tree $(i,J)^\iinfty$ corresponding to the $\pm 1$-twisted order $n$ Whitney disk 
$W_{(i,J)}$.


Having arranged that all the order $2n-1$ trees in $t(\cW)$ occur in algebraically canceling pairs,
applying the tree-simplification and geometric cancellation described in \cite{ST2} to all these algebraically canceling pairs 
yields an order $2n$ twisted Whitney tower $\cW'$.

\textbf{The even order case:} For $\cW$ of order $2n$ with $\tau^\iinfty_{2n}(\cW)=0\in\cT^\iinfty_{2n}$, we arrange for $t(\cW)$ to
consist of only algebraically canceling pairs of generators by
realizing all relators in $\cT_{2n}^\iinfty$,  then construct an order
$2n+1$ twisted Whitney tower by introducing a new method
for geometrically cancelling the pairs of twisted Whitney disks 
(while the algebraically canceling pairs of non-$\iinfty$ trees lead to geometrically canceling intersections as before):

First of all, the order $0$ case corresponding to linking numbers is easily checked, so we will assume $n\geq 1$.

The IHX relators and AS relations for non-$\iinfty$ trees can be realized as usual.

Note that any signed tree $\epsilon\cdot J^\iinfty\in t(\cW)$ does not depend on the orientation of the tree $J$ because changing the orientation on the corresponding twisted Whitney disk $W_J$ does not change $\omega (W_J)$.

For any rooted tree $J$ the relator $\langle J,J \rangle-2\cdot J^\iinfty$ corresponding to the 
interior-twist relation can be realized as follows. Use finger moves to create a clean
framed Whitney disk $W_J$. Performing a positive interior twist on $W_J$
as in Figure~\ref{InteriorTwistPositiveEqualsNegative-fig}  creates a
self-intersection $p\in W_J\cap W_J$ with $t_p=\langle J,J \rangle$ and
changes the twisting $\omega(W_J)$ of $W_J$ to $-2$. 
The negative of the relator is similarly constructed starting with a negative twist.
\begin{figure}[h]
\centerline{\includegraphics[width=125mm]{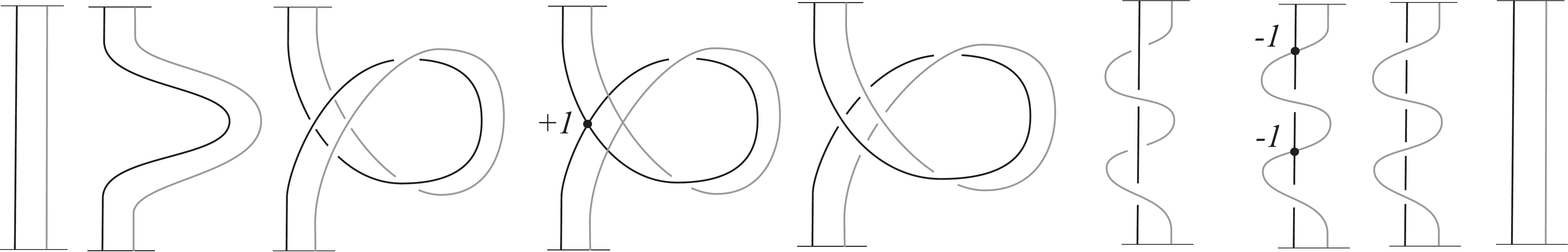}}
         \caption{A $+1$ interior twist on a Whitney disk changes the twisting by $-2$, as is seen in the pair of $-1$ intersections between the black vertical slice of the Whitney disk and the grey slice of a Whitney-parallel copy. Note that the pair of (positive) black-grey intersections near the $+1$ intersection is just an artifact of the immersion of the normal bundle into $4$--space and does not contribute to the relative Euler number.}
         \label{InteriorTwistPositiveEqualsNegative-fig}
\end{figure}

The relator $-I^\iinfty+H^\iinfty+X^\iinfty-\langle H,X \rangle$ corresponding to
the twisted IHX relation is realized as follows. For any
rooted tree $I$, create a clean framed Whitney disk $W_I$ by
finger moves. Then split this framed Whitney disk using the twisted
finger move of Lemma~\ref{lem:split-w-tower} into two clean twisted
Whitney disks with twistings $+1$ and $-1$, and associated signed
$\iinfty$-trees $+I^\iinfty$ and $-I^\iinfty$, respectively. The next step is
to perform a $+1$-twisted version (described in Lemma~\ref{lem:twistedIHX} below)
of the ``Whitney move IHX'' construction of Lemma~{7.2}
in \cite{S1}, which will replace the $+1$-twisted Whitney disk by two
$+1$-twisted Whitney disks having $\iinfty$-trees $+H^\iinfty$ and $+X^\iinfty$,
and containing a single negative intersection point with tree
$-\langle H,X \rangle$, where $H$ and $X$ differ locally from $I$ as in the usual IHX relation.
Thus, any Whitney tower can be modified to create
exactly the relator $-I^\iinfty+H^\iinfty+X^\iinfty-\langle H,X \rangle$, for any rooted tree $I$.
The negative of the relator can be similarly realized by using Lemma~\ref{lem:twistedIHX}
applied to the
$-1$-twisted $I$-shaped Whitney disk.

So since $\tau_{2n}^\iinfty(\cW)$ vanishes, it may be arranged, by realizing relators as above, that
all the trees in $t(\cW)$ occur in algebraically canceling pairs.
Now, by repeated applications of Lemma~\ref{lem:twistedIHX} below, the algebraically canceling pairs of clean $\pm 1$-twisted Whitney disks
can be exchanged for (many) algebraically canceling pairs of clean $\pm 1$-twisted Whitney disks,
all of whose trees are \emph{simple} (right- or left-normed), with the $\iinfty$-label at one end of the tree as illustrated in Figure~\ref{simple-infty-tree-fig} -- this also creates more algebraically canceling pairs of non-$\iinfty$ trees (the ``error term'' trees in
Lemma~\ref{lem:twistedIHX}).

As in the odd case, all algebraically canceling pairs of
intersections with
non-$\iinfty$ trees can be exchanged for geometrically canceling pairs by \cite{ST2}.
To finish building the desired 
order~$2n+1$ twisted Whitney tower, we will describe how to eliminate the remaining algebraically canceling pairs of clean twisted order~$n$ Whitney disks (all having simple trees)
using a construction that bands together Whitney disks and is additive on twistings.
This construction is an iterated elaboration of a construction originally from Chapter~10.8 of \cite{FQ} (which was used to show that
that $\tau_1 \otimes \Z_2$ did not depend on choices of pairing intersections by Whitney disks).

 \begin{figure}[h]
\centerline{\includegraphics[width=75mm]{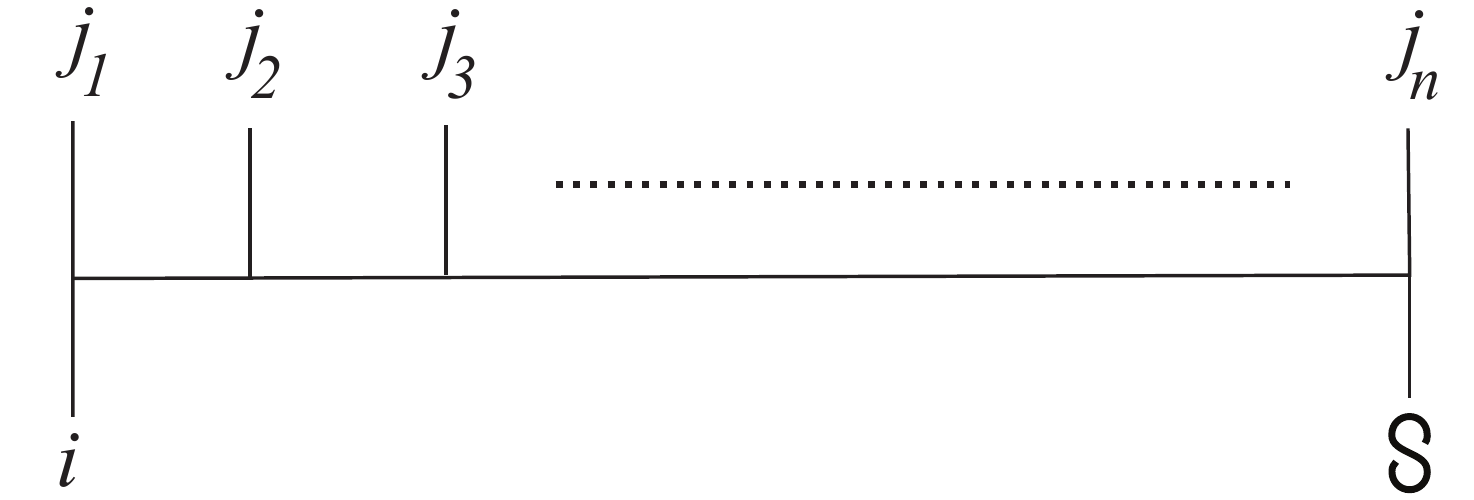}}
         \caption{The simple twisted tree $J_n^\infty$.}
         \label{simple-infty-tree-fig}
\end{figure}

Consider an algebraically canceling pair of clean $\pm 1$-twisted Whitney disks $W_{J_n}$ and $W'_{J_n}$, whose simple $\iinfty$-trees $+ J_n^\iinfty$ and $-J_n^\iinfty$ are as in
Figure~\ref{simple-infty-tree-fig}, using the notation $J_n=(\cdots ((i,j_1),j_2),\cdots , j_n)$. Each trivalent vertex corresponds to a Whitney disk, and we will work from left to right, starting with the order one Whitney disks $W_{(i,j_1)}$ and $W'_{(i,j_1)}$,
banding together Whitney disks of the same order from the two trees, while only creating new unpaired intersections of order greater than $2n$. At the last step, $W_{J_n}$ and $W'_{J_n}$ will be banded together into a single framed clean Whitney disk, providing the desired geometric cancellation. (The reason for working with \emph{simple} trees is that the construction for achieving geometric cancellation requires \emph{connected}
surfaces for certain steps. For instance, the following construction only gets started because the left most trivalent vertices
of an algebraically canceling pair of simple trees correspond to Whitney disks which pair the connected order zero surfaces $D_i$ and $D_{j_1}$.)

To start the construction consider the Whitney disks $W_{(i,j_1)}$ and $W'_{(i,j_1)}$, pairing intersections between the order zero immersed disks $D_i$ and $D_{j_1}$. Let $V$ be another Whitney disk for a canceling pair consisting of one point from each of the points paired by
$W_{(i,j_1)}$ and $W'_{(i,j_1)}$. Figure~\ref{bandedWdisksWithTwistsNoHigherInts} illustrates how a parallel copy $V'$ of $V$ can be banded together
with $W_{(i,j_1)}$ and $W'_{(i,j_1)}$ to form a Whitney disk $W''_{(i,j_1)}$ for the remaining canceling pair. The twisting of $W''_{(i,j_1)}$ is the sum of the twistings on $W_{(i,j_1)}$, $W'_{(i,j_1)}$, and $V$; so $W''_{(i,j_1)}$ is framed if $V$ is framed, since both $W_{(i,j_1)}$ and $W'_{(i,j_1)}$ are framed for $n>1$ (and in the $n=1$ case $W_{(i,j_1)}=W_{J_n}$ and $W'_{(i,j_1)}=W'_{J_n}$ contribute canceling $\pm 1$ twistings). If $V$ is both framed and clean, then the result of
replacing $W_{(i,j_1)}$ and $W'_{(i,j_1)}$ by $V$ and $W''_{(i,j_1)}$ preserves the order of $\cW$ and creates no new intersections.

So if $n=1$, then $W_{J_n}$ and $W'_{J_n}$ have been geometrically canceled, meaning that their corresponding $\iinfty$-trees have been eliminated from $t(\cW)$ without creating any new unpaired order $2n$ intersections or new twisted order $n$ Whitney disks.

\begin{figure}[h]
\centerline{\includegraphics[width=115mm]{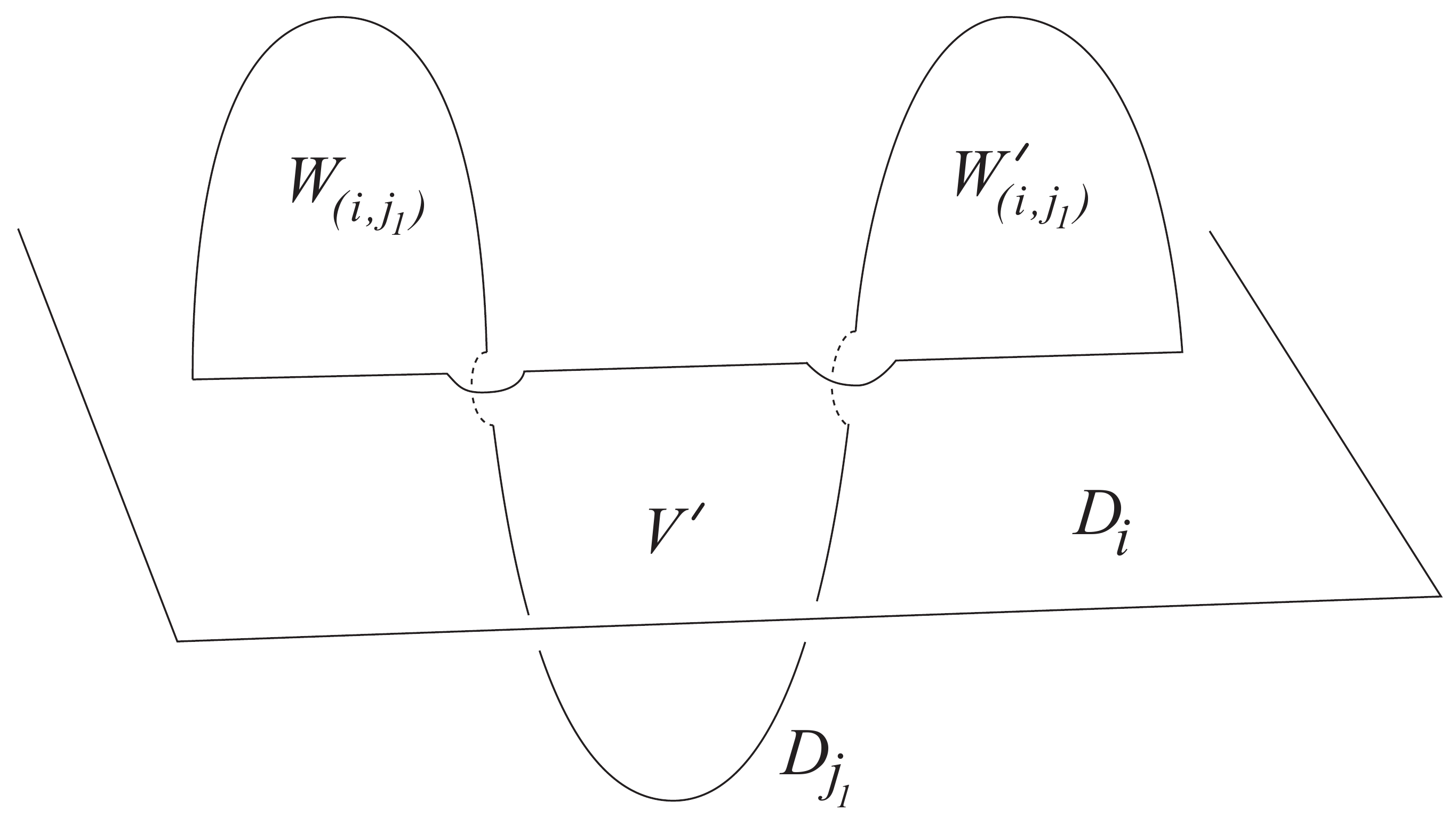}}
         \caption{The Whitney disks $W_{(i,j_1)}$, $W'_{(i,j_1)}$, and $V'$ are 
         banded together to form the Whitney disk
         $W''_{(i,j_1)}$ pairing the outermost pair of intersections between $D_i$ and $D_{j_1}$.
         In the cases $n>1$, the interior of $W''_{(i,j_1)}$ contains two pairs of canceling intersections with $D_{j_2}$ (which are not shown), and supports the sub-towers consisting of the rest of the higher-order Whitney disks (that were supported by $W_{(i,j_1)}$ and $W'_{(i,j_1)}$) corresponding to the trivalent vertices in both trees
$\pm J_n^\infty$.}
         \label{bandedWdisksWithTwistsNoHigherInts}
\end{figure}

The next step shows how $V$ can be arranged to be framed and clean, at the cost of only creating intersections of order greater than $2n$: Any twisting $\omega(V)$ can be killed by boundary twisting $V$ into $D_{j_1}$.  Then, using the construction shown in Figure~\ref{bandedWdisksWithPushDown}, any interior intersection between $V$
and any $K$-sheet (e.g.~an intersection with $D_{j_1}$ from boundary-twisting) can be pushed down into $D_i$ and paired by a thin Whitney disk $W_{(K,i)}$, which in turn has intersections with the $D_{j_1}$-sheet that can be paired by a Whitney disk
$W_{K_1}:=W_{((K,i),j_1)}$ made from a Whitney-parallel copy of $W_{(i,j_1)}$. 
Now, parallel copies of the Whitney disks from the sub-tower supported by $W_{(i,j_1)}$ can be used to build a sub-tower on $W_{K_1}$:
Using the notation $K_{r+1}=(K_r,i)$,
for $r=1,2,3,\ldots n$, the Whitney disk $W_{K_{r+1}}$ is built from a Whitney-parallel copy of $W_{J_r}$, and pairs intersections between $W_{K_r}$ and $j_r$. Note that the order of each $W_{K_r}$ is at least $r$.
The top order $W_{K_{n+1}}$ inherits the $\pm 1$-twisting from $W_{J_n}$, and has a single interior intersection with tree $\langle K_{n+1}, J_n \rangle$ which is of order at least $2n+1$. 

For multiple intersections between $V$ and various $K$-sheets this part of the construction can be carried out simultaneously using nested parallel copies of the thin Whitney disk in 
Figure~\ref{bandedWdisksWithPushDown} and more Whitney-parallel copies of the sub-towers described in the previous paragraph.

The result of the construction so far is that the left-most trivalent vertices of the trees $+ J_n^\iinfty$ and $-J_n^\iinfty$ now correspond to the \emph{same} order $1$ Whitney disk $W''_{J_1}=W''_{(i,j_1)}$, at the cost of having created (after splitting-off) a clean twisted Whitney disk of order at least $n+1$, and an unpaired intersection of order at least $2n+1$. In particular, this completes the proof for the case $n=1$.

\begin{figure}[h]
\centerline{\includegraphics[width=115mm]{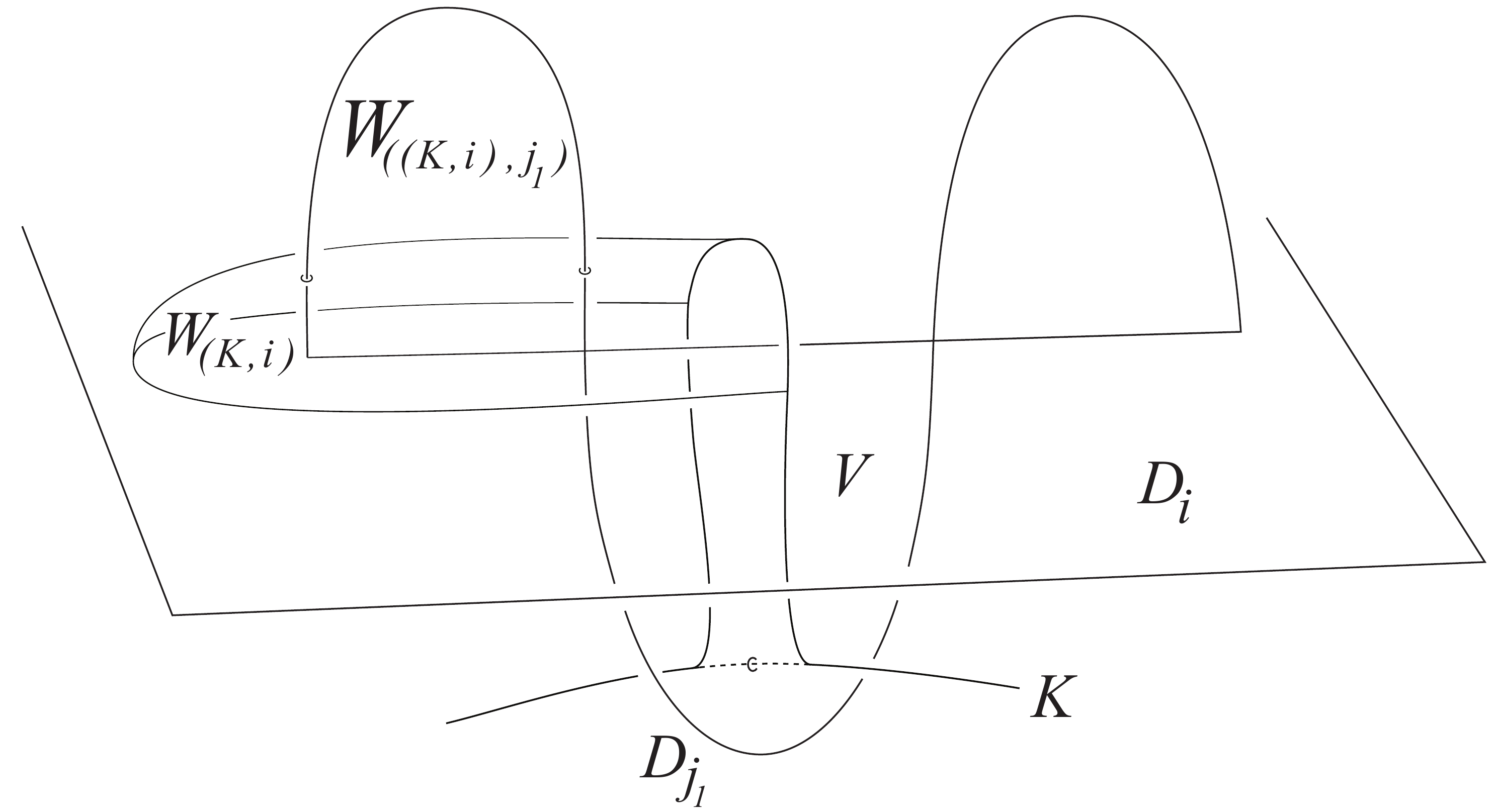}}
         \caption{}
         \label{bandedWdisksWithPushDown}
\end{figure}

For the cases $n>1$, observe that since $W''_{J_1}$ is \emph{connected}, this construction can be repeated, with $W''_{J_1}$ playing the role of $D_i$, and $D_{j_2}$ playing the role of $D_{j_1}$, to get a single order $2$ Whitney disk $W''_{J_2}$ which corresponds to the second trivalent vertices
from the left in both trees $+ J_n^\iinfty$ and $-J_n^\iinfty$. By iterating the construction, eventually we band together $W_{J_n}$ and $W'_{J_n}$ into a single framed clean Whitney disk at the last step, having created only clean twisted Whitney disks of order at least $n+1$, and unpaired intersections of order at least $2n+1$.

\subsection{The twisted IHX lemma}\label{subsec:twistedIHX}

The proof of Theorem~\ref{thm:twisted-order-raising-on-A} is completed by the following lemma which describes a
twisted geometric IHX relation
based on the framed version in Lemma~{7.2}
of \cite{S1}.

\begin{lem}\label{lem:twistedIHX}
Any split twisted Whitney tower $\cW$ containing a clean $+1$-twisted Whitney disk with signed $\iinfty$-tree
$+I^\iinfty$ can be modified (in a neighborhood of the Whitney disks and local order zero sheets corresponding to $I^\iinfty$)
to a twisted Whitney tower $\cW'$ such that $t(\cW')$ differs from $t(\cW)$ exactly by replacing
$+I^\iinfty$ with the signed trees $+H^\iinfty$, $+X^\iinfty$, and $-\langle H,X \rangle$, where $+I-H+X$ is a Jacobi relator.

Similarly, a clean $-1$-twisted Whitney disk with $\iinfty$-tree $-I^\iinfty$ in $t(\cW)$
can be replaced by $-H^\iinfty$, $-X^\iinfty$, and $+\langle H,X \rangle$ in $t(\cW')$.
\end{lem}

\begin{proof}
Before describing how to adapt the construction and notation of
\cite{S1} to give a detailed proof of Lemma~\ref{lem:twistedIHX}, we explain why the framed relation
$+I=+H-X$ leads to the twisted relation $+I^\iinfty=+H^\iinfty+X^\iinfty-\langle H,X \rangle$. In the framed case, a Whitney disk
with tree $I$ is replaced by Whitney disks with trees $H$ and $X$, such that the new Whitney disks are parallel
copies of the original using the Whitney framing, and inherit the framing of the original. In order to preserve the trivalent
vertex orientations of the trees, the orientation of the H-Whitney disk is the same as the original 
I-Whitney disk, and the
orientation of the X-Whitney disk is the opposite of the I-Whitney disk. Now, if the original I-Whitney disk was 
$+1$-twisted,
then both the H- and X-Whitney disks will inherit this same $+1$-twisting, because the twisting -- which is a self-intersection number -- is independent of the Whitney disk orientation. 
The H- and X-Whitney disks will also intersect in a single point with sign $-1$,
since they inherited opposite orientations from the I-Whitney disk. Thus, (after splitting) a twisted Whitney tower can be modified so that a $+I^\iinfty$ is replaced by exactly $+H^\iinfty+X^\iinfty-\langle H,X \rangle$ in the intersection forest. Similarly, a $-I^\iinfty$ can be replaced exactly by $-H^\iinfty-X^\iinfty+\langle H,X \rangle$.

The framed IHX Whitney-move construction is described in detail in \cite{S1} (over four pages, including six figures).
We describe here how to adapt that construction to the present twisted case, including the relevant modification of notation.
Orientation details are not given in \cite{S1}, but all that needs to be checked is that the X-Whitney disk inherits the opposite
orientation as the H-Whitney disk (given that the tree orientations are preserved, and using our negative-corner orientation convention in \ref{subsec:w-tower-orientations} above). In Lemma~{7.2}
of \cite{S1}, the ``split sub-tower $\cW_p$''  refers to the Whitney disks and order zero sheets containing the tree $t_p$ of an unpaired intersection $p$ in a split Whitney tower $\cW$. In the current setting, a clean $+1$-twisted Whitney disk $W$ plays the role of $p$, and the construction will modify $\cW$ in a neighborhood of the Whitney disks and order zero sheets containing the $\iinfty$-tree associated to $W$.
In the notation of Figure~18 of \cite{S1}, the sub-tree of the I-tree denoted by $L$ contains $p$, so to interpret the entire construction in our case only requires the understanding that this sub-tree contains the $\iinfty$-label sitting in $W$. (Note that in Figure~18 of \cite{S1} the labels $I$, $J$, $K$ and $L$ denote \emph{sub}-trees, and in particular the $I$-labeled sub-tree should not be confused with the ``I-tree'' in the IHX relation.)

In the case where the $L$-labeled sub-tree is order zero, then $L$ is just the $\iinfty$-label, and the upper trivalent vertex of the I-tree in Figure~18 of \cite{S1} corresponds to the clean $+1$-twisted $W$, with $\iinfty$-tree $((I,J),K)^\iinfty$. Then the construction, which starts by performing a Whitney move on the framed Whitney disk
$W_{(I,J)}$ corresponding to the lower trivalent vertex of the I-tree, yields the $+1$-twisted
H- and X-Whitney disks as discussed in the first paragraph of this proof, with $\iinfty$-trees $(I,(J,K))^\iinfty$ and $(J,(I,K))^\iinfty$, and non-$\iinfty$ tree $\langle (I,(J,K)),(J,(I,K))\rangle$ corresponding to the resulting unpaired intersection (created by taking Whitney-parallel copies of the twisted $W$ to form the H- and X-Whitney disks).

In the case where the $L$-labeled sub-tree is order $1$ or greater, then the upper trivalent vertex of the
I-tree in Figure~18 of \cite{S1} corresponds to a framed Whitney disk, and Whitney-parallel copies of this framed Whitney disk and the other Whitney disks corresponding to
the $L$-labeled sub-tree are also used to construct the sub-towers containing the $+1$-twisted Whitney disks with H and X $\iinfty$-trees (which will again will lead to a single unpaired intersection as before).
\end{proof}

\subsection{Twisted even and framed odd order Whitney towers.}\label{subsec:boundary-twisted-IHX-lemma}
To complete the proof of Proposition~\ref{prop:exact sequence} in the introduction, the following lemma implies that 
$\bW^\iinfty_{2n}\subset \bW_{2n-1}$:
\begin{lem}\label{lem:boundary-twisted-IHX}
If a collection $A$ of properly immersed surfaces in a simply connected $4$--manifold supports an order $2n$ \emph{twisted} Whitney tower, then $A$ is homotopic (rel $\partial$) to 
$A'$ which supports an order $2n-1$ \emph{framed} Whitney tower.
\end{lem}
\begin{proof}
Let $\cW$ be any order $2n$ twisted Whitney tower $\cW$ supported by $A$. If $\cW$ contains no order $n$ non-trivially twisted Whitney disks, 
then $\cW$ is an order $2n$ framed Whitney tower, hence also is an order $2n-1$ framed Whitney tower. If $\cW$ does contain order $n$ non-trivially twisted Whitney disks, they can be eliminated at the cost of only creating intersections of order at least $2n-1$ as follows:

Consider an order $n$ twisted Whitney disk $W_J\subset\cW$ with twisting $\omega(W_J)=k\in\Z$. If $W_J$ pairs intersections between an order zero surface $A_i$ and an order $n-1$ Whitney disk $W_I$ then $J=(i,I)$, and by performing $|k|$ boundary-twists of $W_J$ into $W_I$, $W_J$ can be made to be framed at the cost of only creating $|k|$ order $2n-1$ intersections, whose corresponding trees are of the form $\langle\,(i,I),I\,\rangle$.

If $W_J$ pairs intersections between two Whitney disks, then by applying the twisted geometric IHX move of
Lemma~\ref{lem:twistedIHX} as necessary, $W_J$ can be replaced by a union of order $n$ twisted Whitney disks
each having a boundary arc on an order zero surface as in the previous case, at the cost of only creating unpaired intersections of order $2n$, each of which is an error term in Lemma~\ref{lem:twistedIHX}.   
\end{proof}

\section{Proof of Theorem~\ref{thm:odd}}\label{sec:proof-thm-odd}
This section defines the doubling map $\Delta_{2n-1}:\Z_2\otimes\cT_{n-1}\rightarrow \cT_{2n-1}$ which determines the framing relations described in the introduction, and strengthens the obstruction theory for framed Whitney towers described in \cite{ST2} by showing that the vanishing
of $\tau_{2n-1}(\cW)$ in the reduced group $\widetilde{\cT}_{2n-1}:=\cT_{2n-1}/\im(\Delta_{2n-1})$ is sufficient for the promotion of $\cW$ to a Whitney tower of order 
$2n$. This means that $\cT_n$ can be replaced everywhere by 
$\widetilde{\cT}_n$ (with $\widetilde{\cT}_{2n}:=\cT_{2n}$) throughout Section~\ref{sec:realization-maps}, and in particular proves Theorem~\ref{thm:odd} of the introduction. 

\begin{defn}\label{def:Delta}
The \emph{doubling map} $\Delta_{2n-1}:\Z_2\otimes\cT_{n-1}\rightarrow \cT_{2n-1}$ is defined for generators $t\in\cT_{n-1}$ by
$$\Delta (t):=\sum_{v\in t} \langle i(v),(T_v(t),T_v(t))\rangle$$
where $T_v(t)$ denotes the rooted tree gotten by replacing $v$ with a root, and the sum is over all univalent vertices of $t$, with $i(v)$ the original label of the univalent vertex $v$. 
\end{defn}
That $\Delta_{2n-1}$ is well-defined as a homomorphism on $\cT_{n-1}$ is clear since AS and IHX relations go to doubled relations. The image of $\Delta_{2n-1}$ is $2$-torsion by AS relations and hence it factors through $\Z_2\otimes\cT_{n-1}$.
See Figure~\ref{fig:Delta trees} for explicit illustrations of $\Delta_1$ and $\Delta_3$.

\begin{figure}[h]
\centerline{\includegraphics[width=115mm]{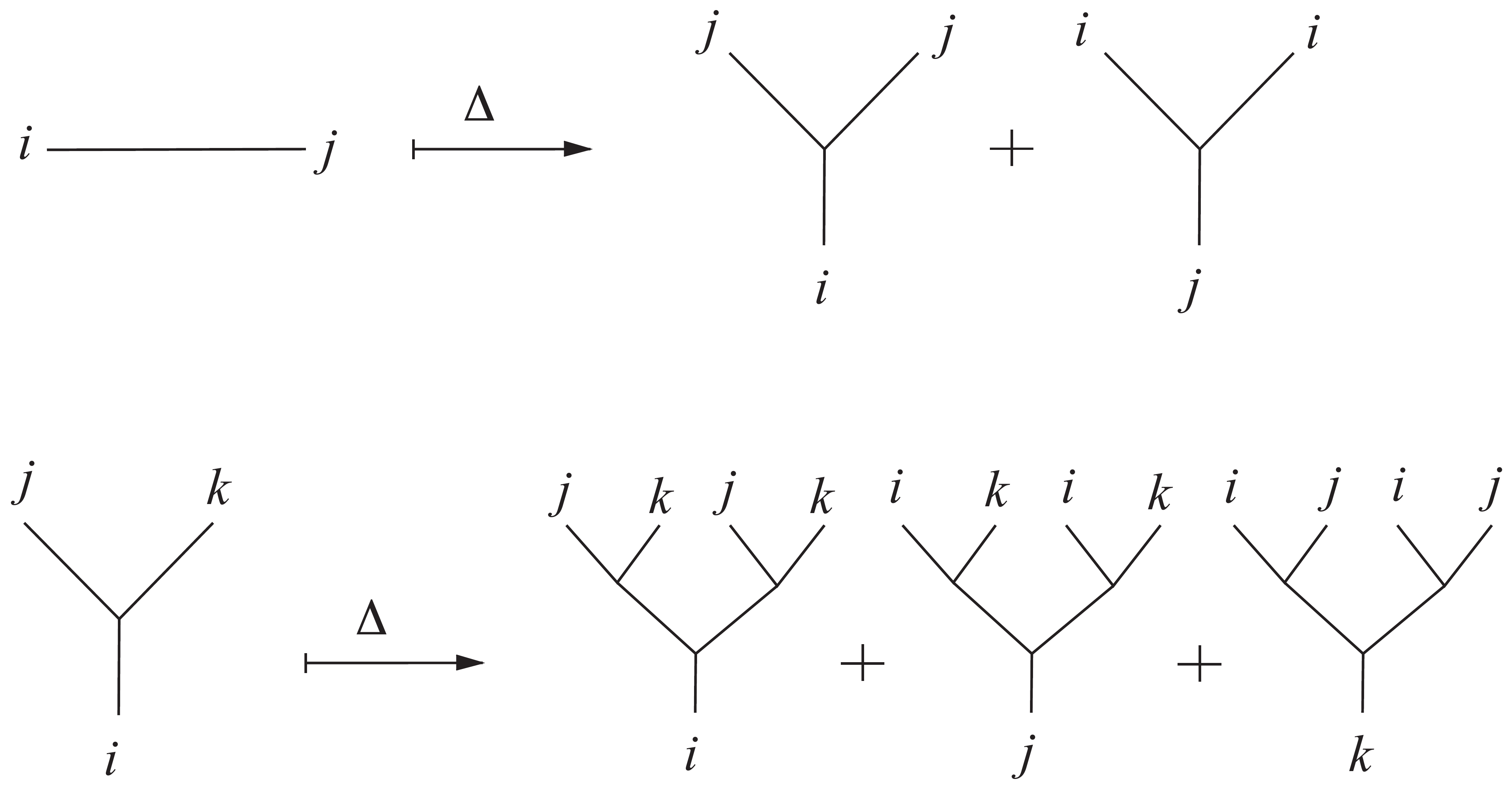}}
         \caption{The map $\Delta_{2n-1}:\Z_2\otimes\cT_{n-1}\rightarrow \cT_{2n-1}$ in the cases $n=1$ and $n=2$.}
         \label{fig:Delta trees}
\end{figure}

The following theorem strengthens Theorem~\ref{thm:framed-order-raising-on-A} in Section~\ref{sec:w-towers}.
\begin{thm}\label{thm:framed-order-raising-mod-Delta}
If a collection $A$ of properly immersed surfaces in a simply connected $4$--manifold supports a framed
 Whitney tower $\cW$ of order $(2n-1)$ with $\tau_{2n-1}(\cW)\in\im(\Delta_{2n-1})$, then $A$ is homotopic (rel $\partial$) to  $A'$ which supports a framed Whitney tower of order $2n$.
\end{thm}

\begin{proof}
As discussed above in the outline the proof of Theorem~\ref{thm:twisted-order-raising-on-A} (just after subsubsection~\ref{subsubsec:alg-vs-geo-cancellation}), to prove Theorem~\ref{thm:framed-order-raising-mod-Delta} it will suffice to show that the intersection forest $t(\cW)$ can be changed by trees representing any element in
$\im(\Delta_{2n-1})<\cT_{2n-1}$ at the cost of only introducing trees of order greater than or equal to $2n$, so that the order $2n-1$ trees in $t(\cW)$ all occur in algebraically canceling pairs. Note that $\im(\Delta_{2n-1})$ is $2$-torsion by the AS relations, so orientations and signs are not an issue here. As in Section~\ref{sec:proof-twisted-thm}, elements of $t(\cW)$ will be denoted by formal sums, and $\cW$ will not be renamed as modifications are made.

\textbf{The case $n=1$:} Given any order zero tree $\langle i,j \rangle$, create a clean framed Whitney disk 
$W_{( i,j )}$ by performing a finger move between the order zero surfaces $A_i$ and $A_j$. Then use a twisted finger move 
(Figure~\ref{twist-split-Wdisk-fig}) to split $W_{( i,j )}$ into two twisted Whitney disks with associated trees $(i,j)^\iinfty -(i,j)^\iinfty$. 
Now boundary-twist each Whitney disk into a different sheet to recover the framing and add 
\[
\langle i,(i,j) \rangle +  \langle j,(i,j) \rangle = \Delta_1( \langle i,j \rangle )
\]
 to $t(\cW)$. Alternatively, after creating the framed $W_{( i,j )}$, perform an interior twist on $W_{( i,j )}$ to get 
$\omega(W_{( i,j )})=\pm2$, then kill $\omega(W_{( i,j )})$ by two boundary-twists, one into each sheet, again adding $\langle i,(i,j) \rangle +  \langle j,(i,j) \rangle$ to $t(\cW)$.
Note that $\im\Delta_1$ in $\cT_1$ corresponds to the order $1$ FR framing relation of \cite{S3,ST1}.

\begin{figure}[h]
\centerline{\includegraphics[scale=.35]{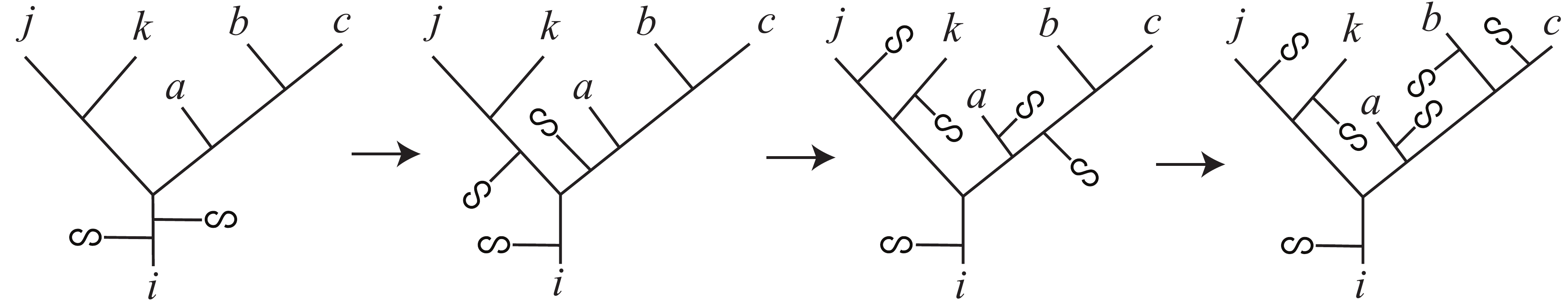}}
         \caption{Multiple $\infty$-roots attached to a tree represent sums (disjoint unions) of trees. On the left: the two trees that result from twist-splitting a clean $W_{( i, (I_1, I_2))}$ in the case $\langle i,(I_1, I_2) \rangle=\langle i, ((j,k),(a,(b,c))) \rangle$. Each arrow indicates an application of a twisted IHX Whitney move, which pushes $\infty$-roots towards the univalent vertices. The right-most sum of trees becomes the image of $\langle i,(I_1, I_2) \rangle$ under $\Delta$ after applying boundary-twists to the associated twisted Whitney disks.}
         \label{fig:Delta-infty-tree-example}
\end{figure}
\textbf{The cases $n>1$:} For any order $n-1$ tree $\langle i,(I_1, I_2) \rangle$, create a clean $W_{( i, (I_1, I_2))}$ by finger moves. (Here we are taking any order $n-1$ tree, choosing an $i$-labeled univalent vertex, and writing it as the inner product 
of the order zero rooted tree $i$ and the remaining order $n-1$ tree.)
Then split $W_{( i, (I_1, I_2))}$ using a twisted finger move to get two twisted Whitney disks each having associated $\iinfty$-tree
$( i, (I_1, I_2))^\iinfty$. 
Leave one of these twisted Whitney disks alone, and to the other apply the twisted geometric IHX Whitney-move (Lemma~\ref{lem:twistedIHX} of Section~\ref{sec:proof-twisted-thm}) to replace 
$( i, (I_1, I_2))^\iinfty$ by $( I_1, (I_2, i))^\iinfty+( I_2, (i,I_1))^\iinfty-\langle ( I_1, (I_2, i)),( I_2, (i,I_1))\rangle$ in $t(\cW)$. Note that the 
tree $\langle ( I_1, (I_2, i)),( I_2, (i,I_1))\rangle$ is order $2n$. If $I_1$ and $I_2$ are not both order zero then continue to 
apply the twisted geometric IHX Whitney-move (pushing the $\iinfty$-labeled vertices away from the $\iinfty$-labeled vertex that is adjacent to the original $i$-labeled vertex) until the resulting union of trees has all $\iinfty$-labeled vertices adjacent to a univalent vertex (all twisted Whitney disks have a boundary arc on an order zero surface) -- see Figure~\ref{fig:Delta-infty-tree-example} for an example.  Then, boundary-twisting each twisted Whitney disk into the order zero surface recovers the framing on each Whitney disk and the resulting change in $t(\cW)$ is a sum of trees as in the right hand side of the equation in Definition~\ref{def:Delta} representing the image of 
$\langle i,(I_1, I_2) \rangle$ under $\Delta_{2n-1}$, together with trees of order at least $2n$.
\end{proof}



\begin{thebibliography}{blah} 



\bibitem{Can}{\bf J W Cannon}, {\it The recognition problem: what is a topological manifold?}, 
Bull. AMS 84 (1978) 832--866. 





\bibitem{C} {\bf T Cochran}, { \it Derivatives of links, Milnor's concordance invariants 
and Massey products}, Mem. Amer. Math. Soc. Vol. 84 No. 427 
(1990). 

\bibitem{C1} {\bf T Cochran}, {\it $k$-cobordism for links in $S^3$}, Trans. Amer. Math. Soc. 327 no 2  (1991) 641--654.


\bibitem{COT} {\bf T Cochran}, {\bf K Orr}, {\bf P\ Teichner}, {\it Knot 
concordance, Whitney towers and $L^2$-signatures}, 
Annals of Math., Volume 157 (2003) 433--519. 

\bibitem{CT1} {\bf J Conant}, {\bf P\ Teichner}, { \it Grope coborbism of classical knots}, 
Topology 43  (2004) 119--156. 

\bibitem{CT2} {\bf J\ Conant}, {\bf P\ Teichner}, {\it Grope Cobordism and Feynman Diagrams}, 
Math. Annalen 328 (2004) 135--171. 


\bibitem{CST} {\bf J Conant}, {\bf R Schneiderman}, {\bf P\ Teichner}, {\it Jacobi
identities in low-dimensional topology}, Compositio Mathematica
143 Part 3 (2007) 780--810.

\bibitem{CST0}  {\bf J Conant}, {\bf R Schneiderman}, {\bf P\ Teichner}, {\it Higher-order intersections in low-dimensional topology}, Proc. Natl. Acad. Sci. USA 2011 108 (20) 8081--8084


\bibitem{CST2} {\bf J Conant}, {\bf R Schneiderman}, {\bf P\ Teichner}, {\it Milnor Invariants and Twisted Whitney Towers}, preprint (2010) 	
arXiv:1102.0758v1 [math.GT].

\bibitem{CST3} {\bf J Conant}, {\bf R Schneiderman}, {\bf P\ Teichner}, {\it Tree homology and a conjecture of Levine,} preprint (2010) http://arxiv.org/abs/1012.2780


\bibitem{CST4}  {\bf J Conant}, {\bf R Schneiderman}, {\bf P\ Teichner}, {\it Universal quadratic refinements and
untwisting Whitney towers}, preprint (2010) http://arxiv.org/abs/1101.3480

\bibitem{CST5} {\bf J Conant},  {\bf R Schneiderman}, {\bf P\ Teichner}, {\it 
Geometric filtrations of string links and homology cylinders}, preprint (2012) arXiv:1202.2482v1 [math.GT].

\bibitem{WTCCL} {\bf J Conant}, {\bf R Schneiderman}, {\bf P\ Teichner}, {\it Whitney tower concordance of classical links}, preprint (2012) arXiv:1202.3463v1 [math.GT].




\bibitem{FQ}  {\bf M\ Freedman}, {\bf F Quinn}, {\it The topology of 
$4$--manifolds}, Princeton Math. Series 39 Princeton, NJ, (1990). 

\bibitem{FT2}  {\bf M\ Freedman}, {\bf P Teichner}, { \it 4--manifold topology II: 
Dwyer's filtration and surgery kernels}, Invent.  Math. 122 
(1995) 531--557. 


 
\bibitem{Gu1} {\bf M Goussarov}, {\it On $n$-equivalence of knots and invariants of finite degree},  
Topology of manifolds and varieties,   Adv. Soviet Math. 18, Amer. Math. Soc. Providence RI (1994) 173--192.


\bibitem{Gu2} {\bf M Goussarov}, {\it Variations of knotted graphs. The geometric technique of $n$-equivalence},  
(Russian)  Algebra i Analiz  12  (2000),  no. 4, 79--125;  translation in  St. Petersburg Math. J.  12  (2001),  no. 4, 569--604.

\bibitem{H} {\bf K Habiro}, {\it Claspers and finite type
invariants of links},  Geom. Topol.  4 (2000) 1--83.


%





%

%


%




\bibitem{L2}  {\bf J Levine}, { \it Addendum and correction to: 
Homology cylinders: an enlargement of the mapping class group}, 
Alg. and Geom. Topology 2 (2002) 1197--1204. 

\bibitem{L3} {\bf J Levine}, {\it Labeled binary planar trees and quasi-Lie algebras}, 
Alg. and Geom. Topology 6 (2006) 935--948. 












\bibitem{S1} {\bf R\ Schneiderman}, { \it Whitney towers and Gropes in 4--manifolds}, 
Trans. Amer. Math. Soc. 358 (2006), 4251--4278. 

\bibitem{S2} {\bf R\ Schneiderman}, { \em Simple Whitney towers, half-gropes and the Arf invariant of a 
knot}, Pacific Journal of Mathematics, Vol. 222, No. 1, Nov 
(2005). 

\bibitem{S3} {\bf R\ Schneiderman}, {\it Stable concordance of knots in $3$--manifolds}, 
Alg. and Geom. Topology 10 (2010) 37--432. 



\bibitem{ST1} {\bf R\ Schneiderman}, {\bf  P\ Teichner}, {\it Higher order intersection numbers of $2$--spheres in $4$--manifolds}, 
Alg. and Geom. Topology 1 (2001) 1--29. 

\bibitem{ST2} {\bf R\ Schneiderman}, {\bf P\ Teichner}, 
{ \em Whitney towers and the Kontsevich integral}, Proceedings of 
a conference in honor of Andrew Casson, UT Austin 2003, Geometry 
and Topology Monograph Series, Vol. 7 (2004), 101--134. 



\bibitem{Sc} {\bf A. Scorpan}, {\it The Wild World of 4-Manifolds}, American Mathematical Society (2005).



\bibitem{T1} \textbf{P Teichner}, {\em Knots, von Neumann Signatures, and Grope 
Cobordism.}\\ 
Proceedings of the International Congress of Math. Vol II: Invited 
Lectures (2002) 437--446. 

\bibitem{T2}  \textbf{P Teichner}, {\em What is ... a grope?}\\ 
Notices of the A.M.S. Vol. 54 no. 8 Sep 2004, 894--895. 


\end{thebibliography}
\end{document}